\documentclass[10pt,a4paper, twoside]{amsart}
\usepackage[T1]{fontenc}
\usepackage[utf8]{inputenc}
\usepackage[english]{babel}
\usepackage{amsmath}
\usepackage{amsthm}
\usepackage{amssymb}
\usepackage{amsfonts}
\usepackage{amsxtra}
\usepackage{enumerate}
\usepackage{verbatim}
\usepackage{color}
\usepackage[mathscr]{eucal}
\usepackage{enumerate}

\let\phi=\varphi
\let\e=\varepsilon
\let\d=\delta

\let\subset=\subseteq

\newcommand{\N}{\mathbb N}
\newcommand{\Z}{\mathbb Z}
\newcommand{\R}{\mathbb R}

\newcommand{\T}{\mathfrak{T}}
\newcommand{\LF}{\mathfrak{L}}

\newcommand{\id}{\text{id}}

\newtheorem{satz}{Satz}[section]

\newtheorem{theorem}[satz]{Theorem}
\newtheorem{definition}[satz]{Definition}
\newtheorem{lemma}[satz]{Lemma}
\newtheorem{cor}[satz]{Corollary}
\newtheorem{fact}[satz]{Fact}

\newtheorem{prop}[satz]{Proposition}
\newtheorem{remark}[satz]{Remark}

\newtheorem*{lem}{Lemma}
\newtheorem*{prop1}{Proposition 3.7}
\newtheorem*{theo1}{Theorem 4.3}

\DeclareMathOperator*{\supp}{supp}
\DeclareMathOperator*{\conv}{conv}
\DeclareMathOperator*{\diam}{diam}
\DeclareMathOperator*{\dist}{dist}
\DeclareMathOperator*{\pos}{pos}

\DeclareMathOperator*{\fil}{fill}
\DeclareMathOperator*{\err}{err}
\DeclareMathOperator*{\inj}{inj}
\DeclareMathOperator*{\Time}{Time}
\DeclareMathOperator*{\Light}{Light}
\DeclareMathOperator*{\std}{std(g_R)}

\DeclareMathOperator*{\relint}{relint}

\begin{document}

\title{Aubry-Mather Theory and Lipschitz Continuity of the Time Separation}

\author[]{Stefan Suhr}
\address{Fachbereich Mathematik, Universit\"at Hamburg}
\email{stefan.suhr@math.uni-hamburg.de}

\date{\today}

\begin{abstract}
We consider Aubry-Mather theory for a subclass of class A spacetimes, i.e. compact vicious spacetimes with globally hyperbolic Abelian cover. 
In this subclass, called class A$_1$, we obtain improved results on timelike maximizers and Lipschitz continuity of the time separation of the Abelian cover 
on the i.g. optimal subsets.
\end{abstract}
\maketitle

\section{Introduction}\label{c4}

The existence problem of timelike geodesic lines and rays and with it the problem of the existence of timelike limit curves is persistent in Lorentzian geometry. 
For example the Lorentzian splitting theorem (\cite{es}) assumes the existence of a timelike geodesic line. A certain quality of co-rays of this line, namely the 
distance of the tangents to the light cones, is essential to establishing the Lipschitz continuity of the associated Busemann function, which in turn is an important
step in the proof of the splitting theorem. Conditions for the existence of such a line, or less restrictively a ray, are barely known in general situations, i.e. without 
any curvature or completeness assumptions.

The underlying geometric problem of the qualitative behavior of maximal geo\-desics can be studied in globally hyperbolic spacetimes which appear as Abelian 
covers of compact vicious spacetimes, using Aubry-Mather theory. Developing Aubry-Mather theory for class A spacetimes, i.e. compact vicious spacetimes with 
globally hyperbolic Abelian covering, \cite{suh110} establishes the existence of at least one timelike maximizer, i.e. a timelike pregeodesic which lifts to an 
arclength--maximizing one in the Abelian covering, in any class A spacetime. The proof further showed that the tangent curve of this maximizer is uniformly 
bounded away from the light cones, i.e. the tangents of an affine parameterization are contained in a compact subset of the tangent bundle. At some points though 
the properties of class A spacetimes were not sufficient to produce the results that one expects for timelike maximizers, i.e. all timelike maximizers in a reasonable 
subset of all maximizers yield flowlines of the geodesic flow contained in a compact subset. If this is not true, one would expect a minimum of $\dim H_1(M,\R)$-
many ``uniformly timelike'' maximizers. Examples of class A spacetimes suggest that this expectation is true in a large sub-class of class A spacetimes. 
These examples include the Lorentzian Hedlund example, the case of $2$-dimensional class A spacetimes and the conformally flat Lorentzian tori (\cite{suh110}). 

Among common properties that all these examples share, is that they give rise to what we will call a uniform family. A uniform family is a continuous family of 
timelike loops such that the base point evaluation map is a proper surjective submersion. It is immediate that the existence of a uniform family implies viciousness 
(proposition \ref{P5-}). Motivated by this fact we call a spacetime uniformly vicious if it admits a uniform family. 

It is the main idea in the present improvement of Aubry-Mather theory to strengthen the viciousness property of class A spacetimes to uniform viciousness. 
This restriction ensures that the following problem does not appear. By elementary reasons it is clear that the stable time separation $\mathfrak{l}$ is positive on 
the interior $\T^\circ$ of the stable time cone $\T$ (see appendix \ref{A1} for the definitions). In contrast no argument is known showing that the support of a 
maximal invariant measure $\mu$ with rotation class $\rho(\mu)\in \T^\circ$ should by confined to the timelike future pointing vectors for general class A 
spacetimes. 

To capture the problem more precisely note that the obstacle to proving the existence of $\dim H_1(M,\R)$-many geometrically distinct {\it timelike} ma\-ximizers is 
the existence of a timelike maximizer $\gamma$ and two limit measures $\mu_0$, $\mu_1$ of $\gamma$ with $\LF(\mu_0)=0$, i.e. $\rho(\mu_0)\in\partial\T$, and 
$\rho(\mu_1)\in \T^\circ$, i.e. $\LF(\mu_1)>0$. For uniformly vicious class A spacetimes we will exclude the existence of such maximizers and thus obtain the 
existence of $\dim H_1(M,\R)$-many distinct timelike maximizers. 

It was mentioned at the beginning that control over the tangents of co-rays to a timelike ray yields the  Lipschitz continuity of the associated Busemann function. We 
employ the acquired control over the tangent vectors of timelike maximizers and the idea underlying the proof of the Lipschitz continuity of the Busemann functions 
in \cite{gaho} to prove the Lipschitz continuity of the time separation of the  Abelian covers of class A$_1$ spacetimes on the i.g. optimal sets. These sets are 
$\T_\e\setminus B_K(0)$, where $\T_\e:=\{h\in\T|\; \dist(h,\partial \T)\ge \e \|h\|\}$ and $B_K(0)$ is the ball of some radius $K>0$ around $0\in H_1(M,\R)$. The 
optimality of $\T_\e$ is immediate from Minkowski space. The necessity to remove $B_K(0)$ from $\T_\e$ follows from the Lorentzian Hedlund examples in 
\cite{suh110}.

To the knowledge of 
the author, so far no result is known about the global Lipschitz continuity of the time separation of a globally hyperbolic spacetime in this generality.

The article is organized as follows. In section \ref{c41} we will define and discuss uniform viciousness for general spacetimes. We give several examples of 
uniformly vicious spacetimes and vicious spacetimes that aren't uniformly vicious. The section is concluded with a smoothing result for uniform families 
(proposition \ref{P5a}). 

Section \ref{S5.2} then discusses the Aubry-Mather theory for class A$_1$ spacetimes. The main technical step in this section is proposition \ref{P6}, while the 
main result proposition \ref{P7a} establishes the existence of at least $\dim H_1(M,\R)$-many geometrically distinct timelike maximizers. 
\begin{prop1}
Let $(M,g)$ be of class A$_1$. Then there exist $\e>0$ and at least $b$-many maximal ergodic measures $\mu_1,\ldots,\mu_b$ of $\Phi$ such that 
$\{\rho(\mu_k)\}$ is a basis of $H_1(M,\R)$ and 
$$\dist(\supp\mu_k,\Light(M,[g]))\ge \e$$
for all $1\le k\le b$.
\end{prop1}

The Lipschitz continuity of the time separation is the subject of section \ref{S4} with the main result being theorem \ref{T18a}. 
\begin{theo1}
Let $(M,g)$ be of class $A_1$. Then for all $\e>0$ there exist constants $K(\e),L(\e)<\infty$ 
such that $(x,y)\mapsto d(x,y)$ is $L(\e)$-Lipschitz on $\{(x,y)\in \overline{M}\times\overline{M}|\,y-x\in 
\T_\e\setminus B_{K(\e)}(0)\}$.
\end{theo1}

The idea to the proof of theorem 
\ref{T18a} is contained in \cite{gaho} and goes back to \cite{es}. We verify the ``timelike co-ray'' condition from \cite{gaho} with the method of proposition \ref{P6}. 
Then the Lipschitz continuity of the time separation follows in the same manner as in \cite{gaho} the Lipschitz continuity of the Busemann functions. 

We conclude these notes with two appendixes. The first one collects the necessary results from earlier work. The second one discusses the notion of causal 
curves and is intended as a motivation for the definitions in section \ref{c41}. 

\emph{Global assumption:} We assume that the manifolds $M$ are equipped with a fixed complete Riemannian metric $g_R$.

\section{Uniformly Vicious Spacetimes}\label{c41}

Before we discuss uniformly vicious spacetimes, we want to note some facts about vicious spacetimes. This is intended as a motivation for the 
subsequent definition of uniform viciousness.

From this point on we will consider $S^1$ as the factor $\R/\Z$. Since $S^1$ becomes a Lie group (as a factor of $(\R,+)$ by $(\Z,+)$) in this way, the 
sum of $s,t\in S^1$ is naturally defined as $s+t:=\overline{s}+\overline{t}+\Z$, where $\overline{s}$ and $\overline{t}$ are real numbers representing 
$s$ and $t$.

Define on the real line the usual metric structure $(\overline{s},\overline{t})\mapsto |\overline{s}-\overline{t}|$. The projection $\pi_{S^1}\colon \R\to S^1$ 
naturally induces a metric structure on $S^1$, which we will denote by $|.|$ as well, i.e. 
$$|s-t|:=\min\{|\overline{s}-\overline{t}||\;\pi_{S^1}(\overline{s})=s, \pi_{S^1}(\overline{t})=t\}.$$
Set $[t-\e,t+\e]:=\{s\in S^1|\; |s-t|\le \e\}$ and $(t-\e,t+\e):=\{s\in S^1|\; |s-t|<\e\}$ for $\e> 0$ and $t\in S^1$. 

First we want to broaden the notion of timelike curves. 
\begin{definition}
Let $(M,g)$ be a Lorentzian manifold. A curve $\gamma\colon I\to M$ is called {\it essentially timelike} if for each $t\in I$ there exist $\delta>0$ and a 
convex normal neighborhood $U$ of $\gamma(t)$ with $\gamma((t-\delta,t+\delta))\subset U$ such that $(\gamma(\sigma),\gamma(\tau))\in 
I_U$ for all $\sigma,\tau\in (t-\delta,t+\delta)$ .

A loop $\gamma\colon S^1\to M$ is essentially timelike if the curve $\overline{\gamma}:=\gamma\circ \pi_{S^1}\colon \R\to M$ is essentially timelike.
\end{definition}

Note that every essentially timelike curve is causal. It is clear that any timelike curve is essentially timelike (recall that we assumed every timelike curve 
to be smooth with timelike tangents). Note that there is no analog to proposition \ref{P00}, in the sense that a curve is essentially timelike if and only if it 
is causal and the $g_R$-arclength parameterization satisfies $\dot{\gamma}(t)\in \Time(M,[g])$ for almost every $t$. For example, consider in Minkowski 
$3$-space $(\R^3,-dt^2+dx^2+dy^2)$ the curve $\gamma(t):=(t,\cos(t),\sin(t))$. We have $\dot{\gamma}(t)\in \Light(\R^3,[-dt^2+dx^2+dy^2])$ for all $t$ 
and obviously $\gamma$ is not a pregeodesic since its trace is not a straight line. We have $(-dt^2+dx^2+dy^2)(\gamma(t)-\gamma(s),\gamma(t)-
\gamma(s))<0$ for every $s\neq t$. This can be seen via two different ways. The first is rather algebraic and considers the tangents to $\gamma$.
One should be aware that $(-dt^2+dx^2+dy^2)(\dot{\gamma},\dot{\gamma})$ vanishes of third order in $t=0$. The second is geometric and rather 
simple. One knows that $\zeta(b)\in I^+(\zeta(a))$ for every future pointing causal curve $\zeta\colon [a,b]\to M$ that is not a lightlike geodesic. 
Since $\gamma$ lies in Minkowski space and is not a geodesic we know that $\gamma(t)\in I^+(\gamma(s))$ for all $s<t$. In Minkowski space 
this condition is equivalent to the assertion. Therefore $\gamma$ is essentially timelike, but no tangent of $\gamma$ is timelike. 

The definition of essential timelikeness is motivated by the observation following from proposition \ref{P01} that every essentially timelike curve can be 
deformed, with fixed endpoints, into a timelike curve via essentially timelike curves. The same is true for essentially timelike loops (a loop $\gamma$ will 
be called a timelike loop if the curve $\overline{\gamma}:=\gamma\circ\pi_{S^1}\colon \R\to M$ is timelike). Note again that any timelike loop is an 
essentially timelike loop. 

The following fact is an alternative definition of total viciousness. For a discussion see \cite{ms1}.

\begin{fact}\label{P4}
A spacetime $(M,g)$ is vicious if and only if there exists an essentially timelike loop $\gamma_p\colon S^1\to M$ with 
$\gamma_p(0)=p$ for all $p\in M$.
\end{fact}

Now fact \ref{P4} motivates the following definition.
\begin{definition}\label{D4}
A spacetime $(M,g)$ is uniformly vicious if there exists a smooth manifold $\mathcal{M}$ and continuous proper map $H\colon \mathcal{M}\times S^1
\to M$ such that $H|_{\mathcal{M}\times \{0\}}$ is a smooth surjective submersion and the loops $H|_{\{x\}\times S^1}$ are essentially timelike for all 
$x\in \mathcal{M}$. We will call the map $H$ a uniform family.
\end{definition}
The idea behind definition \ref{D4} is that the essentially timelike loops in fact \ref{P4} can be chosen in a continuous manner along any subset of $M$. 
From the definition we immediately obtain:

\begin{prop}\label{P5-}
Any uniformly vicious spacetime is vicious.
\end{prop}

Proposition \ref{P5-} is no longer true if we replace essential timelikeness by causality in definition \ref{D4}. For example consider the Lorentzian metric 
$\cos^2(2\pi x)(dx^2-dy^2)+\sin(2\pi x)dxdy$ on $\R^2$. The metric is obviously invariant under integer-translations. Therefore it induces a Lorentzian 
metric on the quotient $\R^2/\Z^2$. The closed causal curves $[x\equiv const]$ in $\R^2/\Z^2$ foliate the torus, but the spacetime is not vicious.

Examples of uniformly vicious spacetimes include flat tori or more generally spacetimes of the following structure: $(M,g)=(N\times S^1, -f^2dt^2+
\beta dt +h))$, where $f$ is any positive smooth function on $N\times S^1$, $\beta$ is a $1$-form on $N$ and $h$ is  Riemannian metric on $N$, both 
depending smoothly on the $S^1$-coordinate. Another set of example is provided by any Lorentzian metric on $S^{2n+1}$ such that the Hopf fibration is 
timelike. 

\begin{prop}
Assume that $M$ is diffeomorphic either to $T^2$, $K^2$ the Klein bottle or $S^1\times \R$. Then every vicious Lorentzian metric on $M$ is uniformly 
vicious.
\end{prop}

\begin{proof}
We consider the case $M\cong T^2$ only. The case $M\cong K^2$ follows from the case $M\cong T^2$ since the orientation cover of $K^2$ is 
diffeomorphic to $T^2$. Then any uniform family for the lifted metric on $T^2$ gives rise to a uniform family on $K^2$ via the canonical projection 
$T^2\to K^2$. Note that any finite cover of a vicious Lorentzian manifold is vicious again.

The other case follows similarly, since any vicious Lorentzian metric on $S^1\times\R$ gives rise to a partition of $S^1\times\R$ into essentially disjoint 
annuli with smooth timelike boundary curves. The uniform family can then be constructed on each annulus separately. If the construction is carried out 
carefully the local uniform families will join to a global uniform family. 

Let $M\cong T^2$ and $g$ a vicious Lorentzian metric on $M$. We can assume w.l.o.g. that $(M,g)$ is time-oriented, due to the same argument 
reducing the case of $M\cong K^2$ to $M\cong T^2$. Then proposition 4.6 in \cite{suh103} implies that stable time cone of $(M,g)$ has nonempty open 
interior. 

Choose a pair of transversal future pointing timelike vector fields $X,Y$ on $M$ such that the rotation vectors of $X$ and $Y$ have different directions 
in the interior of the stable time cone. This can be easily achieved by choosing a future pointing timelike vector field $X$ and a nonsingular lightlike 
vector field $\overline{Y}$. The existence of $X$ follows from the time-orientability of $(M,g)$ and the existence of $Y$ follows from the time-orientability
of $(M,g)$ and the orientability of $M$ (compare \cite{suh102}). Since $M$ is compact, it is clear that the rotation vector of $X$ lies in the interior of the 
stable time cone. Then $Y:=X+\overline{Y}$ yields a future pointing timelike vector field with rotation vector different from $X$. 

Note that every forward orbit of $X$ intersects every backward orbit of $Y$ infinitely many times (The rotation vectors have different directions). 
Therefore by adjoining forward and backward orbits we can construct a continuous family of timelike loops covering $M$. Simply choose a fundamental 
class $\eta$ which is mapped into the interior of the cone over the roation vectors of $X$ and $Y$. Then there exists a positive multiple $\eta^k$ of 
$\eta$ such that for every $p\in M$ the intersection of the forward orbit of $X$ with the backward orbit of $Y$ through $p$ can be chosen such that the 
resulting timelike future pointing curve represents $\eta^k$. These curves depend continuously on $p$ since we have chosen $X$ and $Y$ to be 
transversal. Denote the constant arclength parameterization of the curve through $p$ on $S^1$ with $\gamma_p$.

Now we can choose $\mathcal{M}=M$ and $H\colon M\times S^1\to M$, $(p,t)\mapsto \gamma_p(t)$. $H$ is clearly continuous, the loops $\gamma_p$
are timelike and $H|_{M\times\{0\}}=\id|_M$, i.e. a smooth surjective submersion.
\end{proof}

It is easy to construct examples of vicious Lorentzian manifolds that are not uniformly vicious. Consider the quotient of Minkowski space 
$(\R^2,dx^2-dy^2)$ by the group of translations $\Gamma:=\Z\cdot (0,1)$. Remove from the quotient the point $[(0,0)]$. The claim is then that 
$M:=\R^2/\Gamma \setminus \{[(0,0)]\}$ together with the induced Lorentzian metric $g$ is vicious, but not uniformly vicious. The viciousness is obvious, 
since $(\R^2/\Gamma,dx^2-dy^2)$ is vicious. The other part in the claim is equally easy to be seen. Assume to the contrary that $(M,g)$ is uniformly 
vicious. Note that all loops in the uniform family represent the same fundamental class in $\pi_1(M)$. In the covering space $\R^2\setminus \Z\cdot 
(0,1)$ the lifts of loops in the uniform family are essentially timelike curves connecting a point $p=(p_1,p_2)\in \R^2\setminus \Z\cdot (0,1)$ with 
$p+(0,k)$ for some $k\in \Z$. If $p_1>|k|$ it is obvious that no causal curve connecting $p$ with $p+(0,k)$ can intersect the halfspace$\{(x,y)|\; x\le 0\}$. 
Therefore the fundamental class of each loop in the uniform family must belong to the subgroup $\{a^n|\;n\in\Z\}$, where $a$ is the fundamental class 
represented by the projections to $M$ of $t\mapsto (x,t)\in \R^2\setminus \Z\cdot (0,1)$, $t\in[0,1]$ and $x>0$. The same argument with $p_1<-|k|$ 
shows that the fundamental class of the loops in the uniform family must belong to the subgroup $\{b^n|\; n\in\Z\}$ generated by the the projections to 
$M$ of the curves $t\mapsto (x,t)\in \R^2\setminus\Z\cdot (0,1)$, $t\in[0,1]$ with $x<0$. Since the fundamental group of $M$ is the free group over the 
two generators $a,b$, we obtain $k=0$. But this is clearly a contradiction, since Minkowski space does not contain any causal loops.

In the rest of the section we want to prove the following ``smoothing'' result for uniform families.
\begin{prop}\label{P5a}
Let $(M,g)$ be a uniformly vicious spacetime. Then there exists a smooth uniform family $H'\colon \mathcal{M}\times S^1\to M$ such that 
$H'|_{\{x\}\times S^1}$ is a timelike loop, i.e. with timelike tangents, for every $x\in \mathcal{M}$.
\end{prop}

The main problem with proving such a statement is that if we use smooth approximations of $H$ by the usual method contained in \cite{hirsch}, 
we run into the risk of losing the property that $H|_{\{x\}\times S^1}$ is an essentially timelike loop. We show through a careful analysis that these
problems are futile. 

Choose a complete Riemannian metric $g_R$ on $M$. Define for $p\in M$ the positive number $\inj(M,g)_p$ as the supremum over all $0<\eta$ such 
that $B_\eta(p)$ is contained in a convex normal neighborhood of $p$ in $(M,g)$ with $g_R$-diameter bounded from above by $1$. The diameter 
condition is there for technical reasons. 

\begin{lemma}\label{L40}
Let $(M,g)$ be uniformly vicious. Then there exists a uniform family $H\colon \mathcal{M}\times S^1\to M$ smooth on a neighborhood of 
$\mathcal{M}\times\{0\}$ and such that the loops $H|_{\{x\}\times S^1}$ are piecewise geodesic for all $x\in \mathcal{M}$.
\end{lemma}

\begin{proof}
Let $H_0\colon\mathcal{M}\times S^1\to M$ be a uniform family. Define 
$$i_x:=\inf_{t\in S^1}\{\inj(M,g)_{H_0(x,t)}\}$$ 
for $x\in \mathcal{M}$.  Note that $i_x$ is lower semicontinuous and 
positive for all $x\in\mathcal{M}$. Therefore we can choose a continuous function $\underline{i}\colon \mathcal{M}\to 
(0,\infty)$ with $\underline{i}(x)\le i_x$ for all $x\in \mathcal{M}$. For $N\in \N$ set 
\begin{align*}
U_N:=\{x\in N|\; \dist\left(H_0(x,s),H_0(x,t)\right)&<(2^{-1}-2^{-N})\underline{i}(x),\\
&\forall s,t\in S^1:\;|s-t|< 2^{-N}\}.
\end{align*}
$\{U_N\}_{N\in\N}$ is an open cover of $\mathcal{M}$ and the condition ``$\dist\left(H_0(x,s),H_0(x,t)\right)
<(2^{-1}-2^{-N})\underline{i}(x)$'' ensures that $\overline{U_N}\subset U_{N+1}$. Choose $W_N
\subset\mathcal{M}$ open with $\overline{U_N}\subset W_N\subset \overline{W_N}\subset U_{N+1}$ for every $N\in \N$. 
Then $\{U_{N+1}\setminus \overline{W_{N-1}}\}_{N\in \N}$ forms an open, locally finite cover of $\mathcal{M}$. Choose
a partition of unity $\{\phi_N\}_{N\in\N}$ subordinate to $\{U_{N+1}\setminus \overline{W_{N-1}}\}_{N\in \N}$. 
We use the following index convention
$$\supp \phi_N\subset U_{N+1}\setminus \overline{W_{N-1}}.$$

Consider the set $\mathcal{B}$ of nonzero $0$-$1$-sequences $\alpha=a_1\,a_2\ldots$ which become constant to $0$ 
eventually. Further consider the subset $\mathcal{B}_N$ of $\mathcal{B}$ whose elements are identically $0$ after 
the $N$-th 
digit. Denote by $\beta_N$ the sequence which is identically $0$ except for the $N$-th digit.
Set $r_1\equiv \id|_{\mathcal{B}_1}$ and for $N\ge 2$ define the following operations $r_N\colon \mathcal{B}_N
\to \mathcal{B}_{N-1}$. For $\alpha\in \mathcal{B}_N\setminus \{\beta_N\}$ define $r_N(\alpha)$ by setting the 
$N$-th digit to $0$. For $\alpha=\beta_N$ set $r_N(\alpha)=\beta_{N-1}$.

For $\alpha\in \mathcal{B}$ choose $N\in\N$ with $\alpha\in \mathcal{B}_N$. Set $\alpha_n:=
r_{n+1}\circ\cdots\circ r_N(\alpha)$ for $n\le N-1$ and $\alpha_n:=\alpha$ for $n\ge N$. 
Note that this definition does not depend on the choice of $N$. Set $\overline{t}_\alpha :=
\sum_{n=1}^\infty a_n 2^{-n}$, where $\alpha=a_1a_2\ldots$, and 
$$\overline{t}_\alpha(x):=\sum_{n=1}^\infty \overline{t}_{\alpha_n}\phi_n(x)\in [0,1).$$
Note that $x\mapsto \overline{t}_\alpha(x)$ is smooth. The numbers $\overline{t}_\alpha(x)$ naturally define 
classes $t_\alpha(x)\in S^1$. For the rest of the proof we will denote real numbers in $\R$ with $\overline{t}$ and 
with $t$ their projections to $S^1$.

Denote with $\omega_N$ the sequence which is identical to $1$ for all digits smaller than or equal to $N$, and $0$ 
everywhere else. Define the successor operation 
$$\alpha=a_1a_2\ldots\in \mathcal{B}_N\mapsto s_N(\alpha)=a'_1a'_2\ldots\in \mathcal{B}_N$$
as follows. If $\alpha=\omega_N$ set $s_N(\alpha)=\omega_N$. If $a_N=0$ set $a'_i=a_i$ for $i\le N-1$ and $a'_N=1$. 
If $\alpha\neq \omega_N$ and $a_N=1$ choose $1<k\le N$ minimal such that $a_i=1$ for all $k\le i\le N$. In this case 
set $a'_i=0$ for $k\le i\le N$, $a'_{k-1}=1$ and $a'_j=a_j$ for $j<k-1$. 

Define the map $H_N\colon U_N\times S^1\to M$ as follows. Let $x\in U_N$. For $\alpha \in \mathcal{B}_N$ define 
the curve $\gamma_{\alpha,x}\colon [\overline{t}_\alpha(x), \overline{t}_{s_N(\alpha)}(x)]\to M$ to be the unique geodesic 
connecting $H_0(x,t_\alpha(x))$ with $H_0(x,t_{s_N(\alpha)}(x))$. Note that $\overline{t}_\alpha(x)\le \overline{t}_{s_N(\alpha)}(x)$ and 
$|t_\alpha(x)-t_{s_N(\alpha)}(x)|\le 2^{-n}$ for $x\in U_n$. Therefore 
$$\dist(H(x,t_\alpha(x)),H(x,t_{s_N(\alpha)}(x)))\le i_x/2$$
and $\gamma_{\alpha,x}$ is well defined. By the definition of the uniform family, the geodesic $\gamma_{\alpha,x}$  
is future pointing timelike if and only if $\gamma_{x,\alpha}$ is non-constant if and only if 
$\overline{t}_{s_N(\alpha)}(x)>\overline{t}_\alpha(x)$. 

Define $\overline{t}_\omega,\overline{t}_\beta\colon \mathcal{M}\to \R$ by setting $\overline{t}_\beta|_{U_{N-1}}:= 
\overline{t}_{\beta_N}|_{U_{N-1}}$ and $\overline{t}_\omega|_{U_{N-1}}:=\overline{t}_{\omega_N}|_{U_{N-1}}$. 
$\overline{t}_{\beta}$ and $\overline{t}_\omega$ are well defined smooth functions, since 
$\overline{t}_{\beta_{N+1},\omega_{N+1}}|_{U_{N-1}}\equiv\overline{t}_{\beta_N,\omega_{N}}|_{U_{N-1}}$, and 
induce smooth functions $t_{\beta},t_\omega \colon \mathcal{M}\to S^1$. Now define $H_N\colon U_N\times S^1\to M$
(note that $r_N(\omega_N)=\omega_{N-1}$)
$$H_N(x,t):=\begin{cases}H_0(x,t)&\text{ if } t\in [t_\omega(x),t_\beta(x)]\text{ and}\\ 
                         \gamma_{\alpha,x}(t)&\text{ if }t\in [t_\alpha(x),t_{s_N(\alpha)}(x)]
\text{ for $\alpha\in\mathcal{B}_N$.}
 \end{cases}$$
$H_N$ is continuous by construction and we have $H_{N+1}|_{U_{N-1}}\equiv H_{N}|_{U_{N-1}}$. Therefore we can define a
map $H'\colon \mathcal{M}\times S^1\to M$ with $H'|_{U_{N-1}}\equiv H_N|_{U_{N-1}}$.  $H'$ satisfies the 
claim of the lemma on $(\mathcal{M}\times S^1)\setminus U$, where $U:=\{(x,t)\in\mathcal{M}\times S^1|\; t\in 
(t_\omega(x),t_\beta(x))\}$.

The reason why we have not altered $H_0|_{U}$ so far, is that we want to retain the property that 
$H_0|_{\mathcal{M}\times\{0\}}$ is a surjective submersion. In order to do so we have to be more careful with 
our construction on $U$. Choose for every $x\in \mathcal{M}$ a geodesically convex normal neighborhood $V_x$ of 
$H_0(x,0)$ with $\diam_{g_R}(V_x)\le 1$ such that $B_{i_x}(H_0(x,0))\subset V_x$. Next choose a smooth map 
$\chi\colon \mathcal{M}\times \{0\}\to \Time(M,[g])$ over $H_0|_{\mathcal{M}\times \{0\}}$, i.e. $\chi(x,0)\in 
\Time(M,[g])_{H_0(x,0)}$ for all $x\in \mathcal{M}$. Consider the (future pointing timelike) geodesic 
$c_x$ starting in $H_0(x,0)$ with direction $\chi(x,0)$ and for $N\in \N$ the set 
$$V_N:=\left\{x\in \mathcal{M}|\; c_x\left(\pm N^{-1}\right)\in I^\mp_{V_x}(H_0(x,t_{\beta,\omega}(x)))
\text{ and }N^{-1}<|\overline{t}_{\beta,\omega}(x)|\right\}.$$
$\{V_N\}_{N\in\N}$ is an open cover of $\mathcal{M}$ and we have $\overline{V_N}\subset V_{N+1}$. Choose open sets 
$Z_N$ with $\overline{V_N}\subset Z_N\subset \overline{Z_N}\subset V_{N+1}$. 
Then $\{V_{N+1}\setminus \overline{Z_{N-1}}\}_{N\in \N}$ forms an open, locally finite covering of $\mathcal{M}$. Choose
a partition of unity $\{\psi_N\}_{N\in\N}$ subordinate to this covering. We again use the index convention
$$\supp \psi_N\subset V_{N+1}\setminus \overline{Z_{N-1}}.$$
Set $\overline{\tau}_x:=\sum_{N=1}^\infty (N+1)^{-1}\psi_N(x)\in(0,1)$ and denote with $\tau_x$ the natural projection to $S^1$. 
For $x\in \mathcal{M}$ denote with 
$$\zeta^+_x\colon [\overline{\tau}_x,\overline{t}_{\beta}(x)]\to M\text{ and } \zeta^-_x\colon 
[\overline{t}_\omega(x),1-\overline{\tau}_x]\to M$$
the unique geodesics connecting $c_x(\overline{\tau}_x)$ with $H_0(x,t_\beta(x))$ and $H_0(x,t_\omega(x))$ with 
$c_x(-\overline{\tau}_x)$. Note that $\zeta^\pm_x$ are futurepointing timelike for all $x\in\mathcal{M}$. Define
$$H''(x,t):=\begin{cases}
            H'(x,t)&\text{ if }t\in [t_\beta(x),t_\omega(x)],\\
	    \zeta^+_x(\overline{t})&\text{ if }\overline{t}\in [\overline{\tau}_x,\overline{t}_{\beta}(x)]\\
	    \zeta^-_x(\overline{t})&\text{ if }\overline{t}\in[\overline{t}_\omega(x),1-\overline{\tau}_x]\text{ and}\\
	    c_x(\overline{t})&\text{ if }\overline{t}\in [0,\overline{\tau}_x]\cup [1-\overline{\tau}_x,1].
\end{cases}$$
$H''|_{\{x\}\times S^1}$ is a piecewise geodesic, essentially timelike loop for all $x\in \mathcal{M}$. Since 
$H''$ coincides with $H_0$ on $\mathcal{M}\times\{0\}$, we know that $H''|_{\mathcal{M}\times\{0\}}$ is a surjective submersion.

By construction we have $\sup_{(x,t)}\dist(H_0(x,t),H''(x,t))\le 1$. Recall that all convex normal neighborhoods 
were assumed to have $g_R$-diameter bounded by $1$. Therefore we have  
$$H''^{-1}(K)\subset H^{-1}(B_1(K))$$
for every compact set $K\subset M$. Consequently $H''$ is a proper map as well.
\end{proof}

Denote with $\Time(M,[g])$ the set of future pointing timelike vectors relative $g$ in $TM$ and with $\Light(M,[g])$ the set of future pointing lightlike 
vectors. Further denote with $\Time(M,[g])^\e$ the set of future pointing vectors $v\in \Time(M,[g])$ such that $\dist(v,\Light(M,[g]))\ge \e |v|$.

\begin{definition}
Let $(M,g)$ be a spacetime and $\e>0$. 

(i) A future pointing curve $\gamma\colon I\to M$ is $\e$-timelike if  
$$\dot{\gamma}(t)\in \Time(M,[g])^\e_{\gamma(t)}$$
for one (hence every) $g_R$-arclength parameterization $\gamma\colon \widetilde{I}\to M$ and almost all $t\in \widetilde{I}$. 

(ii) A past pointing curve is said to be $\e$-timelike if it is $\e$-timelike for the reversed time-orientation.
\end{definition}

This is readily extended to general Lorentzian manifolds.
\begin{definition}
(i) A causal curve in a general Lorentzian manifold is said to be $\e$-timelike if one (hence every) lift to the timeorientation cover is $\e$-timelike in the 
lifted metric. 

(ii) A loop $\gamma\colon S^1\to M$ is $\e$-timelike if the lift $\overline{\gamma}\colon [0,1]\to M$ is $\e$-timelike.
\end{definition}
The definition of $\e$-timelikeness is independent of the Riemannian metric in the following sense.

\begin{fact}
Let $g_R,\widetilde{g}_R$ be equivalent Riemannian metrics on $M$, i.e. there exist $0<c\le C<\infty$ such that 
$c\widetilde{g}_R\le g_R\le C\widetilde{g}_R$. Then every, relative to $\widetilde{g}_R$, $\widetilde{\e}$-timelike curve 
is $\frac{c}{C}\widetilde{\e}$-timelike relative to $g_R$. 
\end{fact}

\begin{proof}
Let $\gamma$ be a causal curve in $(M,g)$. Denote the arclength parameter relative to $\widetilde{g}_R$ with 
$\widetilde{s}$ and the one relative to $g_R$ with $s$. Then the parameter change $\phi \colon \widetilde{s}\mapsto s$ 
is a bi-Lipschitz map and $\frac{d}{d\widetilde{s}}\gamma=\frac{d}{ds}\gamma \frac{d}{d\widetilde{s}}\phi$ almost 
everywhere. Since 
$$\dist\nolimits^{g_R}(.,\Light(M,[g]))\le C\dist\nolimits^{\widetilde{g}_R}(.,\Light(M,[g]))\text{ and }|.|^{g_R}\ge c|.|^{\widetilde{g}_R},$$ 
the claim follows immediately.
\end{proof}

Now we can state the analog of proposition \ref{P00} for $\e$-timelike curves.

\begin{prop}
Given a compact subset $K$ of a spacetime $M$ and a future pointing curve $\gamma\colon I\to K$. Then 
$\gamma$ is $\e$-timelike for some $\e>0$ if and only if $\exp^{-1}_{\gamma(s)}(\gamma(t))\in\Time(M,[g])^\delta$
for some $\d>0$ and all $s<t\in I$ sufficiently close.
\end{prop}

For the proof we will need the following elementary estimates. There exist constants $0<\widetilde{c},\widetilde{C}<\infty$, depending only on 
$g$ and $g_R$ and $K$, such that 
\begin{equation}\label{E1}
\widetilde{c}|v|\dist(v,\Light(M,[g])_p)\le |g_p(v,v)|\le \widetilde{C} |v|\dist(v,\Light(M,[g])_p)
\end{equation}
for all $p\in K$ and all future pointing $v\in TM_p$. The proof is elementary and can be found in \cite{su}.

\begin{proof}
If we assume $\exp^{-1}_{\gamma(s)}(\gamma(t))\in \Time(M,[g])^\delta$ for some $\d>0$ and all $s<t\in I$ sufficiently 
close, we obtain $\frac{d}{dt}\gamma(s)\in \Time(M,[g])^\d$ for almost all $t$ (w.l.o.g. we can assume that $\gamma$
is parameterized w.r.t. $g_R$-arclength). Therefore $\gamma$ is $\d$-timelike. 

Conversely assume that $\gamma$ is $\e$-timelike for some $\e>0$. Then we have $L^g(\gamma)\ge \e_1
L^{g_R}(\gamma)$ for some $\e_1>0$. This follows from \eqref{E1} for $v=\dot{\gamma}$ and $\e_1:=\sqrt{\widetilde{c}\e}$. 

Now consider $s<t\in I$ such that $\gamma|_{[s,t]}$ is contained in a compact, convex normal neighborhood $U_s$ 
of $\gamma(s)$ such that $(\overline{U}_s,g|_{\overline{U}_s})$ is globally hyperbolic and $\exp^{-1}_{\gamma(s)}$ is 
bi-Lipschitz on $U_s$. Then we have
$$d_{U_s}(\gamma(s),\gamma(t))\ge L^g(\gamma|_{[s,t]})\ge \e_1L^{g_R}(\gamma|_{[s,t]})
\ge \e_1\dist(\gamma(s),\gamma(t)).$$
Since $\exp^{-1}_{\gamma(s)}$ is bi-Lipschitz on $U_s$, we have 
$$\dist(\gamma(s),\gamma(t))\ge \e_2|\exp^{-1}_{\gamma(s)}(\gamma(t))|$$
for some $\e_2>0$ and therefore
\begin{align*}
d^2(\gamma(s),\gamma(t))&=-g(\exp^{-1}_{\gamma(s)}(\gamma(t)),\exp^{-1}_{\gamma(s)}(\gamma(t)))\\
&\le \widetilde{C}|\exp^{-1}_{\gamma(s)}(\gamma(t))|\dist(\exp^{-1}_{\gamma(s)}(\gamma(t)),\Light(M,[g])).
\end{align*}
The claim now follows for $\d:=\frac{\e_1^2\e_2^2}{\widetilde{C}}$.
\end{proof}

\begin{lemma}\label{L42}
Let $V\subset \R^m$ be open, $M$ a submanifold of $V$ and $g$ a time-oriented Lorentzian metric on $V$ such that
the restriction of $g$ to $M$ is Lorentzian as well. Further let $\e>0$,
$\mathcal{N}$ be a smooth $n$-manifold and $H\colon \mathcal{N}\times S^1\to V$ a continuous map 
such that $H|_{\{y\}\times S^1}$ is an $\e$-timelike loop for all $y\in \mathcal{N}$, $H$ is 
smooth on a neighborhood of $\mathcal{N}\times\{0\}$ and the map $H|_{\mathcal{N}\times\{0\}}$ is a submersion.
Then for every $x\in \mathcal{N}$ there exists a neighborhood $U_x$ of $x$ such that for all 
$\delta\in (0,\e)$ there exists a smooth map $\widetilde{H}_{x,\delta}\colon U_x\times S^1\to V$ with 
$\widetilde{H}_{x,\delta}|_{\{y\}\times S^1}$ is $(\e-\delta)$-timelike for all $y\in U_x$,
$|\widetilde{H}_{x,\delta}(y,t)-H(y,t)|\le \delta$ for all $(y,t)\in U_x\times S^1$ and
$|d\widetilde{H}(y,0)-dH(y,0)|\le\d$ for all $y\in U_x$.
\end{lemma}
We will denote with $B^n_r\subset \R^n$ the open ball of radius $r>0$ and center $0\in \R^n$.
\begin{proof}
The statement is local, therefore we can assume $\mathcal{N}\cong \R^n$. Throughout the proof we will identify 
the tangent spaces $TV_p$ with $\R^m$. Both $\R^n$ and $V$ are equipped with the standard scalar product as 
Riemannian metric. For $p\in V$ we will denote with $\Time(\R^m,[g_p])$ the positively 
oriented timelike vectors in $(\R^m,g_p)\cong (TV_p,g_p)$. $\Time(\R^m,[g_p])^\e$ is defined in the obvious way.
W.l.o.g. we can assume that the loops $H|_{\{x\}\times S^1}$ are future pointing for all $x\in \mathcal{N}$.
Note that by assumption $\mathcal{N}$ is connected.

Choose, for given $x\in \R^n$ and $\delta >0$, a real number $0<\eta<\delta$ and a compact neighborhood 
$K\subset \R^n$ of $x$ such that 
\begin{equation}\label{E110}
H(z,t)-H(z,s)\in \Time(\R^m,[g_p])^{\e-\delta}\cup \{0\},
\end{equation}
for all $z\in K$, $s,t\in S^1$ and $p\in V$ such that $|p-H(z,s)|, |s-t|\le\eta$ and 
$H|_{\{z\}\times [s,t]}$ is future pointing. Note that under these assumptions $H(z,t)-H(z,s)=0$ if and only if 
$s=t$, since $H|_{\{z\}\times S^1}$ is causal. This choice is possible since the loops $H|_{\{z\}\times S^1}$
are $\e$-timelike and we can apply the fundamental theorem of calculus to any arclength parameterization of 
$H|_{\{z\}\times S^1}$. 

Choose $\frac{1}{2}>\kappa'>0$ such that $|H(y,s)-H(z,t)|<\frac{\eta}{2}$
for all $(y,s),(z,t)\in K\times S^1$ with $|y-z|,|t-s|<\kappa'$. Further choose 
smooth functions $\phi\colon \R^n\to [0,\infty)$ and $\widetilde{\theta}\colon \R\to [0,\infty)$ with 
$\supp\phi \subset B_1^n$, $\supp \widetilde{\theta}\subset B^1_1$ and $\int_{\R^n}\phi=
\int_\R \widetilde{\theta}=1$. 
For $0<\kappa<\kappa'$ set $\phi^\kappa(x):=\kappa^{-n}\phi(\kappa^{-1}x)$ and $\widetilde{\theta}^\kappa(t):=
\kappa^{-1}\widetilde{\theta}(\kappa^{-1}t)$. Define functions $\overline{\phi}^\kappa\colon\R^n\times \R^n\to [0,\infty)$, 
$(y',y)\mapsto \phi^\kappa(y'-y)$ and $\overline{\theta}^\kappa\colon \R\times \R\to [0,\infty)$, 
$(t',t)\mapsto \widetilde{\theta}^\kappa(t'-t)$. Since we have $\supp\widetilde{\theta}^\kappa \subset 
(-\frac{1}{2},\frac{1}{2})$ (recall $\kappa<\frac{1}{2}$), the function 
$\theta^\kappa\colon S^1\times S^1\to [0,\infty)$, $(t',t)\mapsto \overline{\theta}^\kappa 
(\overline{t}'-\overline{t})$ is well defined, where $\overline{t}$ and $\overline{t}'$ are lifts of $t$ resp. $t'$ with 
$|\overline{t}'-\overline{t}|<1$. Define for $0<\kappa<\kappa'$
\begin{align*}
\widetilde{H}_{x,\kappa}\colon K\times S^1\to V,\;
(y,t)\mapsto \int_{\R^n\times S^1}H(y',t')\overline{\phi}^\kappa(y',y)
\theta^\kappa(t',t)dy'dt'.
\end{align*}

Our goal is to show that the loops $\widetilde{H}_{x,\kappa}|_{\{y\}\times S^1}$ are $(\e-\delta)$-timelike loops for all 
$y\in K$ and $\kappa$ sufficiently small. We have 
$$|\widetilde{H}_{x,\kappa}(y,t)-H(y,t)|\le \int_{\R^n\times S^1}\overline{\phi}^\kappa(y',y) \theta^\kappa(t',t)
|H(y',t')-H(y,t)|dy'dt'\le \frac{\eta}{2}$$
for all $(y,t)\in K\times S^1$ by our assumption above. Recall that, by definition of $\theta$, we have 
$\theta^\kappa(t+\tau,t)=\widetilde{\theta}^\kappa(\overline{\tau})$ for all $t,\tau\in S^1$ and 
$\kappa<\kappa'$, where $\overline{\tau}$ is the unique lift of $\tau$ to $(-\frac{1}{2},\frac{1}{2}]$. Then we have 
\begin{equation}\label{E111}
\begin{split}
&\widetilde{H}_{x,\kappa}(y,t)-\widetilde{H}_{x,\kappa}(y,s)\\
&=\int_{\R^n}\overline{\phi}^\kappa(y',y)\left[\int_{S^1} 
H(y',t')\theta^\kappa(t',t)dt'-\int_{S^1}H(y',s')\theta^\kappa(s',s)ds'\right]dy'\\
&=\int_{\R^n}\overline{\phi}^\kappa(z,y)\int_{S^1} [H(z,t+\tau)-H(z,s+\tau)]
\widetilde{\theta}^\kappa(\overline{\tau})d\tau dz
\end{split}
\end{equation}
for all $s,t\in S^1$ and $y\in K$.

Recall that we have $|H(z,s+\tau)-H(y,s)|<\eta/2$ if $|z-y|$ and $|\tau|<\kappa$. 
Consequently we have $H(z,s+\tau)\in B_{\eta}(\widetilde{H}_{x,\kappa}(y,s))$ and we get
$$H(z,t+\tau)-H(z,s+\tau)\in \Time(\R^m,[g_{\widetilde{H}_{x,\kappa}(y,s)}])^{\e-\delta}\cup\{0\}$$ 
by (\ref{E110}), for all $z\in K$ and $t$ such that $|s-t|\le \eta$ and $H|_{\{z\}\times [s+\tau,t+\tau]}$ is 
future pointing. Using (\ref{E111}) and the fact that $\Time(\R^m,[g_{\widetilde{H}_{x,\kappa}(y,s)}])^{\e-\delta}$ is a convex 
cone, we obtain
$$\widetilde{H}_{x,\kappa}(y,t)-\widetilde{H}_{x,\kappa}(y,s)\in 
\Time(\R^m,[g_{\widetilde{H}_{x,\kappa}(y,s)}])^{\e-\delta}\cup\{0\}.$$
Since $\widetilde{H}_{x,\kappa}$ is smooth and $\Time(V,[g])^{\e-\delta}_{\widetilde{H}_{x,\kappa}(y,s)}$ is closed,
we get 
$$\partial_t \widetilde{H}_{x,\kappa}(y,s)\in \Time(V,[g])^{\e-\delta}\cup\{0\}$$
for all $(y,s)\in K\times S^1$. It is now easy to see that $\partial_t\widetilde{H}_{x,\d}$ is timelike on a sufficiently 
small neighborhood $U_x$ of $x$. 

The only thing left to note is that this approximation procedure applies to any $C^r$-topology, i.e.
the differentials of $H$ at points $(y,t)$ are approximated by the differentials of $\widetilde{H}$ as well. 
This completes the proof.
\end{proof}

At this point we fix a complete Riemannian metric $\mathcal{G}_R$ on $\mathcal{M}$ once and for all.
The following proof is closely oriented on the smoothing technique presented in \cite{hirsch}.

\begin{proof}[Proof of proposition \ref{P5a}]
Let $H_0\colon \mathcal{M}\times S^1\to M$ be a uniform family.
We will reduce the claim to the case that $M$ is a submanifold of some $\R^{m'}$. For this choose an
embedding $F\colon M\to \R^{m'}$ for some $m'\ge 2m$. Consider $\R^{m'}$ to be equipped with the standard Riemannian
metric $\langle.,.\rangle$. Further consider the normal bundle $\pi_N\colon N\to F(M)$ of $F(M)$ and the 
exponential map $\exp^\perp$ restricted to $N$. i.e. $\exp^\perp \colon N\to \R^{m'}$, $v\mapsto \pi_N(v)+v$. 
Choose a smooth function $\e\colon M\to (0,\infty)$ and a neighborhood $V_N$ of the zero section in $N$ such that 
$\exp^\perp|_{V_N}\colon V_N\to \cup_{p\in M}B_{\e(p)}(F(p))=:V$ is a diffeomorphism. Next define on $V$ the Lorentzian 
metric $g':=(\exp^\perp)_\ast(F_\ast g +\langle .,.\rangle|_N)$ together with the time-orientation such that the 
embedding $F\colon (M,g)\to (V,g')$ preserves time-orientation. Note that, by definition, we have 
$$(\pi_N\circ(\exp^\perp)^{-1})_\ast\colon \Time(V,[g'])\to \Time(F(M),[F_\ast g]),$$
and therefore  timelike curves in $(V,g')$ are mapped to timelike curves by $\pi_N\circ(\exp^\perp)^{-1}$. Furthermore 
note that $\pi_N$ is $1$-Lipschitz relative to the Riemannian metrics 
$g'_R:=(\exp^\perp)_\ast(F_\ast g_R +\langle.,.\rangle|_N)$ and $g_R$. Consequently any smooth map 
$H'\colon \mathcal{M}\times S^1\to V$, such that the loops $H'|_{\{x\}\times S^1}$ are timelike, projects to a 
smooth map $H''\colon\mathcal{M}\times S^1\to F(M)$ such that the loops $H''|_{\{x\}\times S^1}$ are timelike. 
We can further choose $H'$ such that $H''$ is still a surjective submersion onto $F(M)$. 
Thus $H''$ is a smooth uniform family on $F(M)$ and consequently induces such a family on $M$.

Choose bounded open neighborhoods $Z_x, W_x$ and $U_x$ of $x$ such that $\overline{Z}_x\subset W_x\subset 
\overline{W}_x\subset U_x$ and lemma \ref{L42} applies to $H|_{U_x\times S^1}$.
Choose a locally finite subcovering $\{Z_i\}_{i\in \N}$ $(Z_i:=Z_{x_i}, W_i:=W_{x_i}, U_i:=U_{x_i})$ of $\mathcal{M}$.
We want to define inductively smooth maps $H'_j\colon \mathcal{M}\times S^1\to V$ with 
\begin{enumerate}[(1)$_j$]
\item $H'_j\equiv H'_{j-1}$ on $(\mathcal{M}\setminus W_j)\times S^1$.
\item $H'_j$ is smooth on $\cup_{i=1}^j Z_i\times S^1$ and $H'_j|_{\{x\}\times S^1}$ is a smooth 
timelike loop for all $x\in \cup_{i=1}^j Z_i$.
\item $H'_j$ is smooth on a neighborhood of $\mathcal{M}\times\{0\}$ and the projection of $H'_j|{\mathcal{M}\times\{0\}}$
to $F(M)$ is a surjective submersion onto $F(M)$.
\end{enumerate}
Since $\{Z_i\}$ is a locally finite cover, the sequence $\{H'_j\}$ converges on compact subsets of $\mathcal{M}\times S^1$
to a smooth map $H'\colon \mathcal{M}\times S^1\to V$ such that the loops $H|_{\{x\}\times S^1}$ are smooth and timelike.
Therefore the only thing left to prove is the existence of a sequence $\{H'_j\}$ satisfying (1)$_j$ -(3)$_j$. for all $j$. 

For $j=0$ we have $H'_0\equiv H$ and there is nothing to prove. By lemma \ref{L40} we can assume that every loop
$H_0|_{\{y\}\times S^1}$ is piecewise geodesic.  Suppose now that $j>0$ and we have smooth maps $H'_i$ satisfying 
$(1)_i-(3)_i$ for $0\le i< j$. We can choose $\e_{j-1}>0$ such that every loop 
$H'_{j-1}|_{\{y\}\times S^1}$ is $\e_{j-1}$-timelike for $y\in W_j$. Consider for $e_{j-1}>\delta >0$ approximations
$\widetilde{H}_{j,\delta}\colon W_j\times S^1\to V$ of $H'_{j-1}$ according to lemma \ref{L42}.

We know that the loops $\widetilde{H}_{j,\delta}|_{\{y\}\times S^1}$ are $(\e_{j-1}-\d)$-timelike for all 
$y\in W_j$. Consequently we can choose $\delta_j>0$ such that 
$v+H'_{j-1}|_{\{y\}\times S^1}$ and $v+\widetilde{H}_{j,\delta}|_{\{y\}\times S^1}$ is $\e_{j-1}/2$-timelike 
for all $y\in W_j$ and all $v\in B^{m'}_{\delta_j}$. 
Choose a partition of unity $\{\lambda_1,\lambda_2\}$ subordinate to $\{W_j,\mathcal{M}\setminus 
\overline{Z}_j\}$. Define 
$$H'_j(y,t):=\lambda_1(y)\widetilde{H}_{j,\d_j}(y,t)+\lambda_2(y)H'_{j-1}(y,t).$$
We have $H'_j|_{[\mathcal{M}\setminus W_j]\times S^1}\equiv H'_{j-1}|_{[\mathcal{M}\setminus W_j]\times S^1}$,
$H'_j|_{[\cup_{i=1}^j Z_i]\times S^1}$ is smooth and $H'_j|_{\{x\}\times S^1}$ is a smooth timelike 
loop for all $x\in \cup_{i=1}^j Z_i\times S^1$. Therefore $H'_j$ satisfies (1)$_j$ and (2)$_j$.
By the assumptions on $H'_{j-1}$ we know that $H'_j$ is smooth in a neighborhood of $\mathcal{M}\times\{0\}$. 
For $\d_j$ sufficiently small, we know that the projection of $H'_j$ to $F(M)$ is a surjective submersion. This is a
consequence of the standard approximation arguments in \cite{hirsch}.

This completes the induction and the proof.
\end{proof}

\section{Lorentzian Aubry-Mather Theory and Class A$_1$ Spacetimes}\label{S5.2}

\begin{definition}\label{D30}
A compact spacetime $(M,g)$ is of class A$_1$ if it is uniformly vicious and the Abelian cover is globally hyperbolic.
\end{definition}

Note that in this case the domain $\mathcal{M}$ of the uniform family $H$ is compact.
\begin{prop}
Any class A$_1$ spacetime is class A.
\end{prop}

\begin{proof}
Clear from proposition \ref{P5-}.
\end{proof}

\begin{prop}\label{P7}
The set of class A$_1$ metrics is open in $Lor(M)$.
\end{prop}

\begin{proof}
The existence of a smooth uniform family is obviously an open condition in $Lor(M)$.
\end{proof}
From now on we will assume that the given uniform family is smooth. The phenomenon that justifies 
a study of the Mather theory of class A$_1$ spacetimes is the content of the following proposition.

\begin{prop}\label{P6}
Let $(M,g)$ be of class A$_1$. Then no $\alpha\in \partial \T^\ast$ is a support function of $\mathfrak{l}$.
\end{prop}

Before we prove proposition \ref{P6}, we have to introduce some terminology.
Recall that we have chosen a Riemannian metric $\mathcal{G}_R$ on $\mathcal{M}$ (Note that in the compact case 
completeness is not a condition). $\mathcal{G}_R$ naturally induces a Riemannian metric on $T\mathcal{M}$.
The projection $\pi_{T\mathcal{M}}\colon T\mathcal{M}\to \mathcal{M}$ then is $1$-Lipschitz relative to the induced metrics.
Next we will define a bundle map ($\mathcal{Z}$ denotes the zero section of $T\mathcal{M}$)
$$X_H\colon (T\mathcal{M}\setminus\mathcal{Z})\times S^1\to T(\mathcal{M}\times S^1)$$ 
over the identity on $\mathcal{M}\times S^1$, where $(T\mathcal{M}\setminus\mathcal{Z})\times S^1$ carries the obvious 
bundle structure over $\mathcal{M}\times S^1$, as follows: 

Consider for $(v,\phi)\in(T\mathcal{M}\setminus\mathcal{Z})\times S^1$ the quadratic form 
$$b_{(v,\phi)}\colon \R^2\to \R,\;(\lambda,\eta)\mapsto H^\ast g(\lambda v+\eta\partial_\phi,\lambda v+\eta\partial_\phi).$$
The equation $b_{(v,\phi)}(\lambda,\eta)=0$ admits nontrivial solutions for all $(v,\phi)\in (T\mathcal{M}\setminus\mathcal{Z})\times S^1$, since $b_{(v,\phi)}$ is 
either indefinite (if and only $rk(H_\ast|_{span\{v,\partial_\phi\}})=2$) or negative semidefinite (if and only if $rk(H_\ast|_{span\{v,\partial_\phi\}})=1$), but not 
negative definite. Note that $rk(H_\ast|_{span\{v,\partial_\phi\}})\ge1$, since $H_\ast(\partial_\phi)$ is always timelike and therefore 
$g(H_\ast(\partial_\phi),H_\ast(\partial_\phi))<0$. Thus we have $b_{(v,\phi)}(0,\eta)<0$ for all $(v,\phi)$ and $\eta\neq 0$. In the case that 
$rk(H_\ast|_{span\{v,\partial_\phi\}})=2$, the set of solutions consists of two transversal one-dimensional subspaces which depend locally Lipschitz on $(v,\phi)$. 

For every $(v,\phi)\in (T\mathcal{M}\setminus\mathcal{Z})\times S^1$ we define $X_H(v,\phi):=v+\eta\partial_\phi$ as the unique vector such that $\eta$ is maximal 
among all solutions $(1,\eta)$ of $b_{(v,\phi)}(1,\eta)=0$. Then $H_\ast(X_H)$ is future pointing, if $rk(H_\ast|_{span\{v,\partial_\phi\}})=2$ and $0$, if 
$rk(H_\ast|_{span\{v,\partial_\phi\}})=1$. $X_H$ is well defined and continuous for all $(v,\phi)\in (T\mathcal{M}\setminus\mathcal{Z})\times S^1$. Note that $X_H$ 
is locally Lipschitz on the set $\{(v,\phi)\in (T\mathcal{M}\setminus\mathcal{Z})\times S^1|\; rk(H_\ast|_{span\{v,\partial_\phi\}})=2\}$. Since $\mathcal{M}$ is 
compact, we can choose $L<\infty$ and $\e>0$ such that $X_H$ is $L$-Lipschitz on the $\e$-neighborhood of $H_{\ast}^{-1}(\Light(M,[g])\cap T^{1,R}M)
\cap(\ker H_\ast)^\perp$. This is due to the fact that $H_\ast(\partial_\phi)$ is timelike.

Denote with $h_H\in \T^\circ$ the homology class of the curves $H|_{\{p\}\times S^1}$. The fact that $h_H\in \T^\circ$ follows with a simple pertubation argument.

\begin{proof}[Proof of proposition \ref{P6}]
Let $h\in \partial \T\setminus \{0\}$ and $\{\lambda_n\}_{n\in\N}$ be a sequence of positive real numbers diverging to 
$\infty$. Choose with proposition \ref{P1} a sequence of future pointing lightlike maximizers 
$\gamma_n\colon [-T_n,T_n]\to M$ with $|\gamma'_n|\equiv 1$ and 
$$\|\gamma_n(T_n)-\gamma_n(-T_n)-\lambda_n h\|\le \err(g,g_R).$$
Consider a lift $\eta_n\colon [-T_n,T_n]\to \mathcal{M}$ of $\gamma_n$ with $\eta'_n\perp \ker H_\ast$ (Recall 
that $H$ is a surjective submersion). Then there
exists a constant $C_H<\infty$, depending only on $H$, such that $\frac{1}{\|H_\ast\|_\infty}\le |\eta'_n|\le C_H$,
where $\|H_\ast\|_\infty$ denotes the $C^0$-norm of $H_\ast$. Since $\mathcal{M}$ is compact we can assume 
that
$$\frac{1}{2T_n}(\eta'_n)_\sharp(\mathcal{L}^1|_{[-T_n,T_n]})\stackrel{\ast}{\rightharpoonup}\mu\in 
C^0(T\mathcal{M})'.$$ 
By the above bound on $|\eta'_n|$ we have $\mathcal{Z}\cap \supp\mu=\emptyset$. Let $v\in\supp\mu$ and define
$x:=\pi_{T\mathcal{M}}(v)$. By perturbing the map $H$ around $\{x\}\times S^1$, 
we can assume that $rk\{H_\ast((v,0)),H_\ast(\partial_\phi)\}=2$ for all $\phi\in S^1$, i.e. $H_\ast(X_H)_{(v,\phi)}\neq 0$. Choose $\e_0>0$ such 
that $H_\ast(X_H)|_{B_{\e_0}(v)\times S^1}$ is future pointing lightlike and $\delta >0$ such that $\mu(B_{\e_0}(v))\ge 2\delta$. Then we have 
$$\frac{1}{2T_n}(\eta'_n)_\sharp(\mathcal{L}^1|_{[-T_n,T_n]})(B_{\e_0}(v))\ge \delta$$
for sufficiently large $n\in\N$. 

Let $(t,\phi)$ be the canonical coordinates on $[-T_n,T_n]\times S^1$.
By construction $X_H$ induces a vector field $X_n$ on $[-T_n,T_n]\times S^1$ through the condition 
$(\eta_n,\id)_\ast(X_n):=X_H$. Necessary properties of $X_n$ then are 
$$d\phi(X_n)|_{[-T_n,T_n]\times\{0\}}=0\text{ and }dt(X_n)>\e_1$$
for some $\e_1>0$. Next consider $\e_2>0$ such that $(\eta_n,\id)([-T_n,T_n]\times [-\e_2,\e_2])\subset 
B_\e(H_{\ast}^{-1}(\Light(M,[g])\cap T^{1,R}M)\cap(\ker H_\ast)^\perp)$. Note that we can choose $\e_2$ independent of 
$n$, since $\mathcal{M}$ is compact. 
Then $X_n|_{[-T_n,T_n]\times [-\e_2,\e_2]}$ is $L$-Lipschitz, since by construction $\eta'_n\perp \ker H_\ast$ and 
$H_\ast(\eta'_n)=\gamma'_n\in \Light(M,[g])\cap T^{1,R}M$. Next define the vector field
$$Y_n\colon [-T_n,T_n]\times S^1\to T([-T_n,T_n]\times S^1),\;(t,\phi)\mapsto X_n+\frac{1}{\sqrt{|t|+1}}\partial_\phi.$$ 
Note that $(H\circ (\eta_n,\id))_\ast(Y_n)$ is always future pointing timelike in $(M,g)$. Consider a maximal solution $\xi_n\colon [\alpha_n,\omega_n]
\to [-T_n,T_n]\times S^1$ of $\xi'_n=Y_n(\xi_n)$ starting in $(0,0)\in [-T_n,T_n]\times S^1$. With the definition of $Y_n$ we have $-\alpha_n,\omega_n 
\in [\e_1 T_n,T_n]$ and therefore $-\alpha_n,\omega_n\ge \frac{1}{\e_2}$ if $T_n\ge \frac{1}{\e_1\e_2}$. Next consider positive integers $k$ such that 
$\frac{1}{\e_2}\le\tau_k^\omega:=k(k+1)/2\le \omega_n$ and $-\frac{1}{\e_2}\ge\tau_k^\alpha:=-k(k+1)/2\ge \alpha_n$. We have $d\phi(\xi'_n(\tau))
\le\left(L+\sqrt{\frac{2}{\e_1}}\right)\frac{1}{k}$ for $|\tau|\ge \tau_k^\omega=-\tau_k^\alpha$ and $\phi(\xi_n(\tau))\in 
[-1/k,1/k]$ and consequently
$$\int_{\tau^\omega_{k}}^{\tau^\omega_{k+1}}d\phi(\xi'_n)d\tau, 
\int_{\tau^\alpha_{k+1}}^{\tau^\alpha_{k}}d\phi(\xi'_n)d\tau\le 2\left(L+\frac{1}{\sqrt{\e_1}}\right)+1.$$ 
Therefore there exists $C_1<\infty$ such that
\begin{equation}\label{E21-}
\int_{\alpha_n}^{\omega_n}d\phi(\xi'_n)d\tau\le C_1\sqrt{\omega_n-\alpha_n}
\end{equation}
for $n$ sufficiently large. Note that the integral $\int_{-1/\e_2}^{1/\e_2}d\phi(\xi')$ is uniformly bounded. Set $\zeta_n:=H\circ (\eta_n,\id)\circ \xi_n 
\colon [\alpha_n,\omega_n]\to M$. By construction, $\zeta_n$ is a future pointing timelike curve. Next we want to estimate the $g$-length of $\zeta_n$. 
Since $H_\ast(\partial_\phi)$ is future pointing timelike, $H_\ast(X_H)$ is future pointing or vanishing and $\zeta'_n(\tau)=H_\ast(\partial_\phi)+
\frac{1}{\sqrt{|t(\xi_n(\tau))|+1}}H_\ast(X_H)$, we have $(|t(\xi_n(\tau))|\le |\tau|)$
$$L^{g}(\zeta_n)\ge \int_{\alpha_n}^{\omega_n}\frac{\sqrt{2|H^\ast g(\partial_\phi,X_H)|)}}{(|t(\xi_n(\tau))|+1)^{1/4}}
d\tau\ge \int_{\alpha_n}^{\omega_n}\frac{\sqrt{2|H^\ast g(\partial_\phi,X_H)|}}{(|\tau|+1)^{1/4}}d\tau.$$
Recall that $H_\ast(X_H)|_{B_{\e_0}(v)\times S^1}$ is future pointing lightlike, i.e. $H^\ast g(\partial_\phi,X_H)\neq 0$
on $B_{\e_0}(v)\times S^1$.
By decreasing $\e_0$ (and with it $\d$) we can assume that there exists $\e_3>0$ such that $|g(H_\ast\partial_\phi,H_\ast X_H)|(\zeta_n(\tau))\ge \e_3$
whenever $\zeta_n(\tau)\in H(B_{\e_0}(\pi(v))\times S^1)$. The average amount of time that $\zeta_n$ intersects $H(B_{\e_0}(\pi(v))\times S^1)$ is 
bounded from below by $\e_1\delta$. Therefore we obtain
\begin{equation}\label{E21a}
L^{g}(\zeta_n)\ge \frac{\sqrt{\e_3}}{2}\e_1\delta\left(1-\frac{\e_1\delta}{2}\right)^{3/4}(\omega_n-\alpha_n)^{3/4}
\end{equation}
for $\e_1\delta(\omega_n-\alpha_n)\ge 2$. 

Since $h_H\in \T^\circ$, by our assumption on $\gamma_n$ and (\ref{E21-}), we obtain
$$\dist\nolimits_{\|.\|}(\zeta_n(\omega_n)-\zeta_n(\alpha_n),\pos\{h_H,h\})\le \err(g,g_R).$$ 
We extend the curves $\zeta_n$, using proposition \ref{3.2}, by uniformly bounded arcs to future pointing curves $\overline{\zeta}_n$ with $h_n:=
\rho(\overline{\zeta}_n)\in \pos\{h_H,h\}$. Equation (\ref{E21-}) shows that 
$$\dist\nolimits_{\|.\|}(h_n,\pos\{h\})\le \frac{C_1}{\sqrt{\omega_n-\alpha_n}}.$$
By theorem \ref{stab2} (ii) there exists $\lambda >0$ such that $\lim_{n\to\infty} h_n =\lambda h$. Since $\mathfrak{l}(h_n)\ge
\frac{L^g(\overline{\zeta}_n)}{2(\omega_n-\alpha_n)}$, for sufficiently large $n$, we obtain, using (\ref{E21a}),
$$\mathfrak{l}(h_n)\ge \frac{\e_4}{(\omega_n-\alpha_n)^{1/4}}$$
for $n$ sufficiently large and some $\e_4>0$, independent of $n$. But then for sufficiently large $n$ there exists $\e_5>0$ with
\begin{equation}\label{E21}
\mathfrak{l}(h_n)\ge\e_5\sqrt{\dist\nolimits_{\|.\|}(h_n,\pos\{h\})}.
\end{equation}

For any support function $\alpha\in \T^\ast$ of $\mathfrak{l}$ we have $\mathfrak{l}(h)\le \alpha(h)$. If we assume $\alpha\in \partial\T^\ast$, there 
exists $h_\alpha\in\partial\T\setminus \{0\}$ with $\alpha(h_\alpha)=0$. Consequently, we would have
$$\mathfrak{l}(h)\le \|\alpha\|^\ast \dist\nolimits_{\|.\|}(h,\pos\{h_\alpha\})$$
for all $h\in\T$. This contradicts equation (\ref{E21}) for a suitable sequence $\{h_{n,\alpha}\}_{n\in\N}$.
\end{proof}

Next we want to discuss some consequences for the Lorentzian Mather theory of class A$_1$ spacetimes.
\begin{prop}\label{P6a}
Let $(M,g)$ be of class A$_1$. Then for every $\e>0$ there exists $\d(\e)>0$ such that
$$\dist(\supp \mu ,Light(M,[g]))\ge \d(\e)$$
for every maximal invariant measure $\mu$ with $\rho(\mu)\in\T_\e$.
\end{prop}

\begin{proof}
Let $\e>0$ be given. Assume that there exists a sequence of maximal measures $\mu_n$ with $\rho(\mu_n)\in\T_\e$ and
$\dist(\supp\mu_n,\Light(M,[g]))\to 0$ for $n\to \infty$. W.l.o.g. we can assume that $\mu_n(T^{1,R}M)=1$ 
for all $n$. Choose a weakly converging subsequence $\mu_{n_k}$ with weak limit $\mu$. Denote with 
$h\in\T_\e\setminus\{0\}$ the rotation vector of $\mu$. Note that $\mu$ is maximal and $\dist(\supp\mu,\Light(M,[g]))=0$.
Consider any support function $\alpha$ of $\mathfrak{l}$ at $h$. By proposition \ref{P6} we have $\alpha\in(\T^\ast)^\circ$.
Since $\mu\in\mathfrak{M}_\alpha$, we know that any $\gamma$ with $\gamma'\subset \supp\mu$ is calibrated by any
calibration re\-presenting $\alpha$ (proposition \ref{P20a}). By proposition \ref{P20-} the set of calibrations representing 
$\alpha$ is nonempty. But then the conclusion $\dist(\supp\mu,\Light(M,[g]))=0$ contradicts proposition \ref{P20+}.
\end{proof}

Recall the following authentic language introduced in from \cite{suh110}. A future pointing maximizer 
$\gamma\colon \R\to M$ is a $\T^\circ$-maximizer if there exist $\lambda_1,\ldots \lambda_{b+1}\ge 0$ and limit 
measures $\mu_1,\ldots ,\mu_{b+1}$ of $\gamma$ such that $\rho(\sum \lambda_i \mu_i)\in \T^\circ$.

\begin{cor}
Let $(M,g)$ be of class A$_1$. Then any limit measure of a $\T^\circ$-maximizer is supported entirely in $\Time(M,[g])$.
\end{cor}

The following result strengthens the statements of proposition \ref{P10} and proposition \ref{P21} for class A spacetimes in the class of class A$_1$ 
spacetimes. 
\begin{prop}\label{P7a}
Let $(M,g)$ be of class A$_1$. Then there exist $\e>0$ and at least $b$-many maximal ergodic measures $\mu_1,\ldots,\mu_b$ of $\Phi$ such that 
$\{\rho(\mu_k)\}$ is a basis of $H_1(M,\R)$ and 
$$\dist(\supp\mu_k,\Light(M,[g]))\ge \e$$
for all $1\le k\le b$.
\end{prop}

\begin{proof}
Fix $\alpha\in(\T^\ast)^\circ$ and consider $D:=\alpha^{-1}(1)\cap\T$. Then, by proposition \ref{P6}, 
$\mathfrak{l}|_{D}$ is a concave function and every support function $\beta\colon \alpha^{-1}(1)\to\R$ of 
$\mathfrak{l}|_{D}$ satisfies $\beta^{-1}(0)\cap D=\emptyset$. This follows from the fact that any affine function 
on $\alpha^{-1}(1)$ has a unique linear extension to $H_1(M,\R)$, and the linear extensions of support functions
of $\mathfrak{l}|_{D}$ are support functions of $\mathfrak{l}$. Next consider the compact convex body 
$$\mathcal{K}:=\{(h,t)|\;h\in D,\;0\le t\le \mathfrak{l}(h)\}$$
in $\alpha^{-1}(1)\times\R$.
Recall that any point $(h,t)\in\mathcal{K}$ is the convex combination of at most $b$ extremal points of
$\mathcal{K}$. Thus for every $(h,t)\in\mathcal{K}$ we can choose extremal points $(h_i,t_i)$ and $\lambda_i\in(0,1]$ such that $(h,t)=
\sum_{i=1}^{b}\lambda_i(h_i,t_i)$. In the case that $t=\mathfrak{l}(h)$, we obtain that $\mathfrak{l}|_{\conv\{h_i\}_{1\le i\le b}}$ is affine and 
$t_i=\mathfrak{l}(h_i)$ for all $i$, since $(h,\mathfrak{l}(h))\in \relint(\conv\{h_i\}_{1\le i\le b})$ and $\mathfrak{l}$ is concave.
Choose any support function $\beta$ of $\mathfrak{l}|_D$ at $h\in\relint(D)$. Then we have $\beta\equiv \mathfrak{l}$
on $\conv\{h_i\}_{1\le i\le b}$. If there exists $1\le i_0\le b$ with $\mathfrak{l}(h_{i_0})=0$ ,we obtain 
$\beta(h_{i_0})=0$ and a  contradiction to our observation that $\beta^{-1}(0)\cap D=\emptyset$ for all support functions
$\beta$ of $\mathfrak{l}|_D$. Therefore any point $(h,\mathfrak{l}(h))\in \mathcal{K}$ with $h\in\relint(D)$ is the convex 
combination of extremal points $(h',\mathfrak{l}(h'))$ of $\mathcal{K}$ with $\mathfrak{l}(h')>0$. 

Choose for every $1\le j\le b$ homology classes $h_j\in \relint(D)$ such that $\{h_j\}_{1\le j\le b}$ is a basis of $H_1(M,\R)$
and support functions $\beta_j$ of $\mathfrak{l}|_D$ at $h_j$. Next choose for every $j$ a set of extremal points 
$\{(h_{j,i},\mathfrak{l}(h_{j,i}))\}_{1\le i\le b_j}$ of $\mathcal{K}$ and $\lambda^{j,i}\in (0,1]$ such that 
$$\sum_i \lambda^{j,i}(h_{j,i},\mathfrak{l}(h_{j,i}))=(h_j,\mathfrak(h_j)).$$
We have seen that every $\beta_j$ is a support function of $\mathfrak{l}|_D$ at $h_{j,i}$ for every $1\le i\le b_j$ as well.
Choose a basis $\{h'_{k}\}_{1\le k\le b}\subset \{h_{j,i}\}_{1\le j\le b,1\le i\le b_j}$ of $H_1(M,\R)$. Fix $1\le k\le b$.
Like in the proof of proposition \ref{P10} we can consider $\lambda_k>0$ maximal among all $\lambda>0$ with 
$(\rho(\mu),\LF(\mu))=\lambda(h'_k,\mathfrak{l}(h'_k))$ for some $\mu\in \mathfrak{M}^1_g$. The preimage 
of $\lambda_k(h'_k,\mathfrak{l}(h'_k))$ under the map $\mu\in\mathfrak{M}^1_g\mapsto (\rho(\mu),\LF(\mu))$ 
is compact and convex in $\mathfrak{M}^1_g$. Therefore it contains extremal points by the theorem of Krein-Milman. 
Every extremal point of this subset is an extremal point of $\mathfrak{M}^1_g$. Since all measures in $\{\nu\in\mathfrak{M}^1_g|\;(\rho(\nu),\LF(\nu))
=\lambda_k(h'_k,\mathfrak{l}(h'_k))\}$ are maximal, the extremal points are maximal ergodic measures. 
Choose a maximal ergodic measure $\mu_k$ with $(\rho(\mu),\LF(\mu))=\lambda_k(h'_k,\mathfrak{l}(h'_k))$.
The unique linear extension $\alpha_k\in H^1(M,\R)$ of $\beta_k$ is a support function of $\mathfrak{l}$ and 
therefore we have $\alpha_k\in (\T^\ast)^\circ$. By our choice we have $\mu_k\in \mathfrak{M}_{\alpha_k}$. 
Then with proposition \ref{P20+} we get $\supp\mu_k\subset \Time(M,[g])$, i.e. there exist $\e_k>0$ with 
$\dist(\supp\mu_k,\Light(M,[g]))\ge \e_k$. Setting $\e:=\min\{\e_k\}$, the proposition follows.
\end{proof}

The proof especially shows that for class A$_1$ spacetimes with $\mathfrak{l}=0$ somewhere on $\partial\T\setminus\{0\}$ ($\mathfrak{l}$ can only 
vanish on $\partial\T$), there exist infinitely many ergodic maximal measures $\mu$ with $\supp\mu \subset\Time(M,[g])$. This follows from the 
observation that if $\mathfrak{l}$ vanishes somewhere on $\partial\T\setminus \{0\}$, the number of extremal points of $\mathcal{K}$ cannot be finite. 
Therefore the following corollary generalizes corollary 4.8 in \cite{su} to a subclass of class A$_1$ spacetimes (note that globally conformally flat 
Lorentzian tori are trivially of class A$_1$).

\begin{cor}\label{C18}
Let $(M,g)$ be of class A$_1$ and assume that $\mathfrak{l}$ vanishes somewhere on $\partial\T\setminus \{0\}$. 
Then there exist infinitely many maximal ergodic measures $\mu$ with 
$$\supp\mu\in\Time(M,[g]).$$
\end{cor}
The set of class A$_1$ spacetimes satisfying the assumptions of the corollary could be rather 
small in the set of all class A$_1$ spacetimes. It is for example possible to approximate (in any $C^k$-topology) 
any flat Lorentzian metric on the $2$-torus by Lorentzian metrics with $\mathfrak{l}|_{\partial\T\setminus\{0\}}>0$. 
In opposition, for Lorentzian $2$-tori the condition $\mathfrak{l}|_{\partial\T\setminus\{0\}}>0$ can be stable under small 
$C^0$-pertubations of the Lorentzian metric.

\section{Lipschitz continuity of the Time Separation}\label{S4}

\begin{prop}\label{T18}
Let $(M,g)$ be of class A$_1$. Then for all $\e>0$ there exist $\delta>0$ and $K<\infty$ such that 
$$\gamma'(t)\in \Time(M,[g])^\delta$$
for all maximizers $\gamma\colon [a,b]\to M$ with $\gamma(b)-\gamma(a)\in \T_\e\setminus B_K(0)$ and all $t\in[a,b]$.
\end{prop}

We obtain the following immediate corollary.
\begin{cor}\label{T19}
Let $(M,g)$ be of class A$_1$. For every $\e>0$ there \\ exists $K=K(\e)<\infty$ such that for every sequence of 
maximizers $\{\gamma_n\}_{n\in\N}$ with $L^{g_R}(\gamma_n)\ge K$ and $\rho(\gamma_n)\in \T_\e$, any limit curve 
of $\{\gamma_n\}_{n\in\N}$ is timelike.
\end{cor}

\begin{proof}[Proof of Proposition \ref{T18}]
Choose $K(\e)$ such that the assumptions $y-x\in \T_\e$ and $\|y-x\|\ge K(\e)$ imply $y\in I^+(x)$ for all $x,y\in \overline{M}$ (proposition 
\ref{3.2}). The idea is to confirm the existence of $\delta(\e)>0$ such that for all $x,y\in\overline{M}$ with $\|y-x\|\ge K(\e)$ and $y-x\in \T_\e$, any future 
pointing maximizer $\gamma\colon [0,T]\to \overline M$ from $x$ to $y$ satisfies $\gamma'(t)\in\Time(M,[g])^{\delta(\e)}$ for all $t\in [0,T]$. Assume to 
the contrary that there exist a sequence of pairs $(x_n,y_n)\in \overline{M}\times \overline{M}$ with $y_n-x_n\in \T_\e$, maximizers $\gamma_n\colon
[0,T_n]\to \overline{M}$ connecting $x_n$ and $y_n$ and parameter values $t_n\in [0,T_n]$ with $\gamma'_n(t_n)\notin\Time(M,[g])^{1/n}$. The 
sequence $y_n-x_n$ cannot have any accumulation points by the choice of $K(\e)$. If there exist points of accumulation $x, y$, the curves $\gamma_n$ 
will accumulate towards a maximal lightlike limit curve. Then since $\T_\e$ is closed, there exists a lightlike maximizer connecting points $x$ and $y$ 
with $y-x\in\T_\e$. This contradicts the choice of $K(\e)$. Consequently the sequence $\{y_n-x_n\}_{n\in\N}$ must be unbounded.

Since $\Light(M,[g])\cap T^{1,R}M$ is $\Phi$-invariant and $\Phi$ is complete as well as continuous, there exists a sequence $0<a_n\to\infty$ such that 
$L^g(\gamma_n|_{[t_n-a_n,t_n]})\le 1$ and 
$$\|\gamma_n(t_n)-\gamma_n(t_n-a_n)\|\le \frac{\e}{4}\|\gamma_n(T_n)-\gamma_n(0)\|\le\frac{1}{4}\dist\nolimits_{\|.\|}(\gamma_n(T_n)-\gamma_n(0),
\partial\T)$$
for all $n\in \N$. This implies $[\gamma_n(t_n-a_n)-\gamma_n(0)]+[\gamma_n(T_n)-\gamma_n(t_n)]=:v_n\in\T_{\frac{3\e}{4}}$. Fix a lift 
$\overline{\gamma}_n\colon[0,T_n]\to\overline{M}$ of $\gamma_n$ and choose $k_n\in H_1(M,\Z)_\R$ such that 
$\dist(\overline{\gamma}_n(T_n),\overline{\gamma}_n(0)+k_n)\le\fil(g,g_R)$ and $\overline{\gamma}_n(0)+k_n\in J^+(\overline{\gamma}_n(T_n))$ 
(Fact \ref{F1}). Then $[\overline{\gamma}_n(t_n-a_n)+k_n]-\overline{\gamma}_n(t_n)\in\T_{\e/2}$ for sufficiently large $n$. Denote with $L_{\e/2}$ the 
Lipschitz constant of $\mathfrak{l}|_{\T_{\e/2}}$. We obtain
\begin{align*}
\mathfrak{l}(v_n)+\mathfrak{l}&(\gamma_n(t_n)-\gamma_n(t_n-a_n))\le\mathfrak{l}(\gamma_n(T_n)-\gamma_n(0))
\stackrel{\text{\ref{T17}}}{\le} d(\overline{\gamma}_n(0),\overline{\gamma}_n(T_n))+\overline{C}(\e)\\
=\;&d(\overline{\gamma}_n(0),\overline{\gamma}_n(t_n-a_n))+d(\overline{\gamma}_n(t_n-a_n),\overline{\gamma}_n(t_n))
+d(\overline{\gamma}_n(t_n),\overline{\gamma}_n(T_n))+\overline{C}(\e)\\
\le\; & d(\overline{\gamma}_n(t_n),\overline{\gamma}_n(t_n-a_n)+k_n)+d(\overline{\gamma}_n(t_n-a_n),\overline{\gamma}_n(t_n))+\overline{C}(\e)\\
\le\;& \mathfrak{l}([\overline{\gamma}_n(t_n-a_n)+k_n]-\overline{\gamma}_n(t_n))+1+2\overline{C}(\e/2)\\
\le\; &\mathfrak{l}(v_n)+L_{\e/2}(\fil(g,g_R)+\std)+1+2\overline{C}(\e/2).
\end{align*}
Consequently
$$\mathfrak{l}(\gamma_n(t_n)-\gamma_n(t_n-a_n))\le L_{\e/2}(\fil(g,g_R)+\std)+1+2\overline{C}(\e)=:C_1(\e).$$ 
From $a_n\to\infty$ and $\mathfrak{l}|_{\T^\circ}>0$ we obtain for $n$ sufficiently large that $w_n:=\gamma_n(t_n)-\gamma_n(t_n-a_n)\notin\T_{\e/2}$. 
Therefore the homology classes $v_n+w_n$ and $v_n$ are linearly independent and we can define an ``almost support'' function $\alpha_n$ of 
$\mathfrak{l}$ as follows. Set $\alpha_n(v_n):=\mathfrak{l}(v_n)$ and $\alpha_n(v_n+w_n):=\mathfrak{l}(v_n+w_n)$. This defines a unique linear 
function $\alpha_n$ on $span\{v_n,w_n\}$. For any $\lambda\in[0,1]$ we have
\begin{align*}
\mathfrak{l}(v_n+w_n)&\ge \mathfrak{l}(v_n+\lambda w_n)-L_{\e/2}\dist\nolimits_{\|.\|}((1-\lambda)w_n,\T)\\
&\ge \mathfrak{l}(v_n+\lambda w_n)-L_{\e/2}\err(g,g_R)
\end{align*}
by proposition \ref{P1}, and consequently
\begin{align*}
\mathfrak{l}(v_n+\lambda w_n)&\le \mathfrak{l}(v_n)+C_1(\e)+L_{\e/2}\err(g.g_R)=: \mathfrak{l}(v_n)+C_2(\e).
\end{align*}
With the definition of $\alpha_n$ and $\mathfrak{l}(v_n)\le \mathfrak{l}(v_n+w_n)$ we obtain $\alpha_n|_{\conv\{v_n,v_n+w_n\}}\ge 
\mathfrak{l}|_{\conv\{v_n,v_n+w_n\}}-C_2(\e)$ and therefore
\begin{equation}\label{E22}
\alpha_n(h)\ge\mathfrak{l}(h)-\|h\|\frac{C_2(\e)}{\min\{\|v_n+\lambda w_n\||\; \lambda \in [0,1]\}}
\end{equation}
for all $h\in\pos\{v_n,v_n+w_n\}$. Now the concavity of $\mathfrak{l}$ and the definition of $\alpha_n$ imply \eqref{E22} for all $h\in span\{v_n,w_n\}
\cap \T$. Choose, using the Hahn-Banach theorem, an extension $\beta_n\colon H_1(M,\R)\to \R$ of $\alpha_n$ such that 
$$\beta_n\ge \mathfrak{l}-\|.\|\frac{C_2(\e)}{\min\{\|v_n+\lambda w_n\||\; \lambda \in [0,1]\}}.$$
Since $v_n+w_n\in \T_\e$ and $\alpha_n(v_n+w_n)=\mathfrak{l}(v_n+w_n)>0$ uniformly in $n$, we obtain that $\|\beta\|^\ast$ is bounded away from 
$0$ and $\infty$, uniformly in $n$.

Choose converging subsequences $\beta_{n_k}\to\beta\in H^1(M,\R)\setminus \{0\}$, $v_{n_k}/\|v_{n_k}\|\to v\in \T_{\e/2}$ and $w_{n_k}/\|w_{n_k}\|\to 
w\in \T$. Since $\T$ contains no linear subspaces, we have $\min\{\|v_n+\lambda w_n\||\; \lambda \in [0,1]\}\to\infty$ for $n\to\infty$. By continuity of 
$\mathfrak{l}$ on $\T^\circ$ we have $\beta(v)=\mathfrak{l}(v)$ and therefore $\beta\in \T^\ast$. Note that we have
$$\beta_n(v_n+w_n)=\mathfrak{l}(v_n+w_n)\le \mathfrak{l}(v_n)+C_1(\e)=\beta_n(v_n)+C_1(\e)$$
and therefore $\beta_n(w_n)\le C_1(\e)$. Thus we get $\beta(w)=0$ and a contradiction to proposition \ref{P6}.
\end{proof}

With this ``compactness'' result we are able to prove the full Lipschitz continuity of the time separation, 
thus generalizing the coarse-Lipschitz theorem in \cite{suh103}.
\begin{theorem}\label{T18a}
Let $(M,g)$ be of class $A_1$. Then for all $\e>0$ there exist constants $K(\e),L(\e)<\infty$ 
such that $(x,y)\mapsto d(x,y)$ is $L(\e)$-Lipschitz on $\{(x,y)\in \overline{M}\times\overline{M}|\,y-x\in 
\T_\e\setminus B_{K(\e)}(0)\}$.
\end{theorem}
A few comments are in order on why the result is optimal for general class A$_1$ spacetimes. The flat torus 
is an example of a class A$_1$ spacetime for which the Lipschitz continuity of the time separation on the 
Abelian cover can not be extended to $\partial J^+$.

Further the Lorentzian Hedlund examples in \cite{suh110} show that the condition ``$\|y-x\|\ge K(\e)$'' is necessary for the Lipschitz 
continuity. Locally, i.e. for $\|y-x\|$ small, there exist $x,y\in \R^3$ with $d(x,y)=0$ and $y-x\in \T_\e$ for some 
$\e>0$. 

\begin{proof}[Proof of theorem \ref{T18a}]
The proof is almost a word by word transcription of the proof of theorem 3.7 in \cite{gaho}. We use the following 
lemma proved in the appendix of \cite{es}.
\begin{lem}
Let $U$ be an open convex domain in $\R^n$ and $f\colon U\to \R$ a continuous function. Assume that for any 
$q\in U$ there is a smooth lower support function $f_q$ at $q$ such that $|df_q^\sharp(q)|\le L$. Then $f$ is Lipschitz 
with Lipschitz constant $L$. 
\end{lem}
For a given $\e>0$ choose $K(\e)<\infty$, $\d(\e)>0$ as in proposition \ref{T18} and let $x,y\in \overline{M}$ with $y-x\in \T_\e\setminus 
B_{2K(\e)}(0)$. Further choose a convex normal neighborhood $V$ of $x$ such that $y-z\in \T_{\e}\setminus 
B_{K(\e)}(0)$ for all $z\in V$. Next choose a maximizer $\gamma\colon [0,T]\to\overline{M}$ connecting $x$ with $y$ 
and $t>0$ such that $\dist(x,\gamma(t))\ge \inj(\overline{M},\overline{g})/2$. By proposition \ref{P6} there exists 
$\eta=\eta(\e)>0$, independent of $x$ and $y$, such that 
$$-(\exp^{\overline{g}}_{\gamma(t)})^{-1}(z)\in\Time(\overline{M},[\overline{g}])^{\d(\e)}_{\gamma(t)}$$
for all $z\in B_\eta(x)$. Consider on $B_\eta(x)$ the function $f_x(z):=d_V(z,\gamma(t))+d(\gamma(t),y)$, 
where $d_V$ is the local time separation of $(V,\overline{g}|_V)$. Note that $f_x$ is smooth on $B_\eta(x)$ with bounded  
differential by proposition \ref{P6} and 
$$d_V(z,\gamma(t))=\sqrt{|g(\exp^{-1}(z),\exp^{-1}(z))|}.$$
By the reverse triangle inequality, $f_x$ is a lower support function of $d(.,y)$ at $x$. This establishes the assumption 
of the lemma and we obtain that the restricted time separation $d(.,y)$ is Lipschitz at $x$ with Lipschitz constant depending 
only on $\e$. Since the same argument can be applied to $d(x,.)$, we obtain the Lipschitz continuity of the time 
separation $d$ on $\{(x,y)\in \overline{M}\times\overline{M}|\; y-x\in \T_\e\setminus B_{K(\e)}(0)\}$.
\end{proof}


\appendix

\section{Requisites}\label{A1}

In this first appendix we collect very briefly the results on Lorentzian Aubry-Mather theory needed in the text. Reference are \cite{suh110} and 
\cite{suh103}.

\begin{fact}[\cite{suh103}]\label{F1}
Let $M$ be compact and $(M,g)$ a vicious spacetime. Then there exists a constant $\fil(g,g_R)<\infty$ such that any two points $p,q\in M$ can be joined 
by a future pointing timelike curve with $g_R$-arclength less than $\fil(g,g_R)$.
\end{fact}

For a manifold $M$ denote with $\overline{M}$ the Abelian cover, i.e. $\overline{M}= \widetilde{M}/[\pi_1(M),\pi_1(M)]$. 

Let $M$ be a compact manifold and $\{k_1,\ldots ,k_b\}$ a base of  $H_1(M,\R)$ consisting of integer classes. Denote with 
$\{\alpha_1,\ldots,\alpha_b\}$ the dual base and choose representatives $\omega_i\in\alpha_i$. For two points $x,y\in \overline{M}$ define $y-x\in 
H_1(M,\R)$ through $\langle\alpha_i,y-x\rangle=\int_{\overline{\gamma}}\overline{\omega}_i$ where $\overline{\gamma}$ is any Lipschitz curve 
connecting $x$ and $y$, and $\overline{\omega}_i$ is the lift of $\omega_i$ to $\overline{M}$. For a curve $\gamma\colon [a,b]\to M$ we define
$\gamma(b)-\gamma(a)$ via a lift to $\overline{M}$.

Consider a compact spacetime $(M,g)$ and a sequence $\gamma_n\colon [a_n,b_n]\to M$ of future pointing curve such that $L^{g_R}(\gamma_n)\to
\infty$. Define $\T^1$ to be the set of accumulation points of $\left(\frac{\gamma_n(b_n)-\gamma_n(a_n)}{L^{g_R}(\gamma_n)}\right)_n$ in 
$H_1(M,\R)$. Denote with $\T$ the cone over $\T^1$. We call $\T$ the stable timecone.

\begin{prop}[\cite{suh103}]\label{P1}
Let $(M,g)$ be a compact and vicious spacetime. Then $\T$ is the unique cone in $H_1(M,\R)$ such that there exists 
a constant $\err(g,g_R)<\infty$ with $\dist_{\|.\|}(J^+(x)-x,\T)\le \err(g,g_R)$ for all $x\in\overline{M}$, where 
$J^+(x)-x:=\{y-x|\;y\in J^+(x)\}$.
\end{prop}

\begin{prop}[\cite{suh103}]\label{3.2}
Let $(M,g)$ be a compact and vicious spacetime. Then for every $R>0$ there exists a constant $0<K=K(R)<\infty$ such that
    $$B_R(q)\subseteq I^+(p)$$
for all $p,q\in \overline{M}$ with $q-p\in\T$ and $\dist_{\|.\|}(q-p,\partial \T)\ge  K$.
\end{prop}

Recall from \cite{suh103} that a compact spacetime $(M,g)$ is of class A if $(M,g)$ is vicious and the Abelian covering space is globally hyperbolic.

\begin{theorem}[\cite{suh103}]\label{stab2}
Let $(M,g)$ be compact and vicious. Then the following statements are equivalent:
\begin{enumerate}[(i)]
\item $(M,g)$ is of class A.
\item We have $0\notin \T^1$. Especially $\T$ contains no linear subspaces.
\item We have $(\T^\ast)^\circ\neq\emptyset$ and for every $\alpha \in (\T^\ast)^\circ$ 
there exists a smooth $1$-form $\omega$ representing $\alpha$ such that $\ker\omega_p$ is a spacelike hyperplane 
in $(TM_p,g_p)$ for all $p\in M$.
\end{enumerate}
\end{theorem}

\begin{theorem}[\cite{suh110}]\label{T17}
Let $(M,g)$ be of class A. Then there exists a unique concave function $\mathfrak{l}\colon \T \rightarrow \R$ such that for 
every $\e >0$ there is a constant $\overline{C}(\e)<\infty$ with
\begin{enumerate}
\item $|\mathfrak{l}(h)-d(x,y)|\le \overline{C}(\e)$ for all $x,y\in \overline{M}$ with $y-x=h\in \T_\e$ and
\item $\mathfrak{l}(\lambda h)=\lambda \mathfrak{l}(h)$, for all $\lambda \ge 0$,
\item $\mathfrak{l}(h'+h)\ge \mathfrak{l}(h')+\mathfrak{l}(h)$ for all $h,h'\in\T$ and
\item $\mathfrak{l}(h)=\limsup_{h'\to h}\mathfrak{l}(h')$ for $h\in \partial \T$ and $h'\in \T$.
\end{enumerate}
We call $\mathfrak{l}$ the {\it stable time separation}.
\end{theorem}

\begin{definition}[\cite{suh110}]
Let $(M,g)$ be a pseudo-Riemannian manifold and $g_R$ a complete Riemannian metric on $M$. We denote the reparameterization of the geodesic 
flow of $(M,g)$ w.r.t. $g_R$-arclength with the pregeodesic flow $\Phi\colon TM\times\R\to TM$. 
\end{definition}
Note that the pregeodesic flow is still a conservative flow, i.e. it is defined through a differential equation of second order on $M$. 
If not noted otherwise pregeodesics are always assumed to be parametrized by $g_R$-arclength. 

Using the properties of the pregeodesic flow we define rotation classes $\rho(\mu)$ for finite $\Phi$-invariant Borel measures $\mu$ by the condition
$\langle \alpha, \rho(\mu)\rangle =\int_{T^{1,R}M} \omega d\mu$ where $\omega$ represents $\alpha \in H^1(M,\R)$ (\cite{suh110}).

For the obvious reasons we restrict all considerations to measures supported in the future pointing causal vectors. 
For compact and vicious spacetimes follows that the set of rotation classes of finite invariant measures is $\T$. If we impose the class A condition 
we obtain that $\mathfrak{l}(h)=\max\{\int_{\Time(M,[g])}\sqrt{|g(v,v)|} d\mu(v)|\; \rho(\mu)=h\}$, where $\Time(M,[g])$ denotes the set of 
future pointing timelike vectors in $(M,g)$. An invariant measure $\mu$ with 
$$\int_{\Time(M,[g])}\sqrt{|g(v,v)|} d\mu(v)= \mathfrak{l}(\rho(\mu))$$
is called a maximal measure.

\begin{prop}\label{P10}[\cite{suh110}]
Let $(M,g)$ be of class A. Then the pregeodesic flow admits at least $\dim H_1(M,\R)$-many maximal ergodic probability measures.
\end{prop}

Let $\alpha\in H^1(M,\R)$. We call a function $\tau\colon \overline{M}\to\R$ $\alpha$-equivariant if $\tau(x+k)=\tau(x)+\alpha(k)$ for all 
$x\in \overline{M}$ and $k\in H_1(M,\Z)$. 

Denote with $\mathfrak{l}^\ast\colon \T^\ast \to \R$ the dual function of the stable time separation, i.e. $\mathfrak{l}^\ast(\alpha)=\min\{\alpha(h)|\;
\mathfrak{l}(h)=1\}$.

\begin{definition}[\cite{suh110}]
Let $\alpha \in (\T^\ast)^\circ$. An $\alpha$-equivariant and Lipschitz continuous function $\tau \colon \overline{M}\to\R$ is a calibration 
representing $\alpha$ if $\tau(\overline{q})-\tau(\overline{p})\ge\mathfrak{l}^\ast(\alpha)d(\overline{p},\overline{q})$ for all 
$\overline{p},\overline{q}\in\overline{M}$ with $\overline{q}\in J^+(\overline{p})$.
\end{definition}

Note that every calibration is automatically a time function, i.e. strictly monotonous along any causal curve.

\begin{prop}[\cite{suh110}]\label{P20-}
Let $(M,g)$ be a class A spacetime, $\omega\in \alpha\in(\T^\ast)^\circ$ and $F\colon \overline{M}\to \R$ a primitive of 
$\overline{\pi}^\ast(\omega)$. Then the function 
$$\tau_\omega\colon \overline{M}\to \R,\; x\mapsto \liminf_{\substack{y\in J^+(x),\\ \dist(x,y)\to\infty}}
  [F(y)-\mathfrak{l}^\ast(\alpha)\, d(x,y)]$$ 
is a calibration representing $\alpha$.
\end{prop}

\begin{definition}[\cite{suh110}]
Let $(M,g)$ be a class A spacetime and $\tau\colon \overline{M}\to\R$ a calibration representing $\alpha$. A future pointing pregeodesic 
$\gamma \colon \R\to M$ is said to be calibrated by the calibration $\tau$ if 
$$\tau(\overline{\gamma}(t))-\tau(\overline{\gamma}(s))=\mathfrak{l}^\ast(\alpha)L^g(\gamma|_{[s,t]})$$
for one (hence every) lift $\overline{\gamma}$ to $\overline{M}$ and all $s<t\in \R$.
\end{definition}

We say that a future pointing pregeodesic $\gamma\colon\R\to M$ is a maximizer if the one (hence any) lift to the Abelian covering space is maximal, 
i.e. $L^{\overline{g}}(\overline{\gamma}|_{[s,t]})=d(\overline{\gamma}(s),\overline{\gamma}(t))$ for all $s\le t\in\R$. Using the definition of a calibration it 
is obvious that any calibrated curve is a maximizer.

We say that a finite Borel measure $\mu$ on $T^{1,R}M$ is a limit measure of the future pointing pregeodesic $\gamma\colon\R\to M$ if there exists a 
sequence of intervals $[a_n,b_n]$ with $b_n-a_n\to \infty$ and a constant $C\in (0,\infty)$ such that 
$$\frac{C}{b_n-a_n}(\gamma')_\sharp(\mathcal{L}^1|_{[a_n,b_n]})\stackrel{\ast}{\rightharpoonup}\mu$$
where the convergence is the weak-$\ast$ convergence in $C^0(T^{1,R}M,\R)'$. By an elementary calculation we see that a limit measure is always a
$\Phi$-invariant measure.

Denote by $\Light(M,[g])$ the set of future pointing lightlike tangents vectors in $TM$. 

For $\alpha\in \T^\ast$ define $\mathfrak{M}_\alpha$ to be the set of invariant measures $\mu$ that maximize 
$$\mu\mapsto \mathfrak{l}^\ast(\alpha)\int_{T^{1,R}M}\sqrt{|g(v,v)|}d\mu(v)-\langle \alpha,\rho(\mu)\rangle.$$
Set $\supp\mathfrak{M}_\alpha:=\cup_{\mu\in\mathfrak{M}_\alpha}\supp \mu$.

\begin{prop}[\cite{suh110}]\label{P20+}
Let $\alpha\in (\T^\ast)^\circ$ and $\tau\colon\overline{M}\to \R$ a calibration representing $\alpha$. Further let $\gamma\colon \R\to M$ be a future 
pointing maximizer calibrated by $\tau$. Then all limit measures of $\gamma$ belong to $\mathfrak{M}_{\alpha}$. Moreover the image of the tangential 
mapping $t\mapsto \gamma'(t)$ can be separated from $\Light(M,[g])$, i.e. there exists $\e=\e(\alpha)>0$ such that $\dist(\gamma'(t),\Light(M,[g])\ge \e$ 
for all $t\in \R$.
 \end{prop}

\begin{prop}[\cite{suh110}]\label{P20a}
For $\alpha\in (\T^\ast)^\circ$ any pregeodesic $\gamma$ with $\gamma'\subset \supp \mathfrak{M}_\alpha$ is calibrated by every calibration 
representing $\alpha$. In particular there exist calibrated curves.
\end{prop}

\begin{cor}[\cite{suh110}]\label{P21}
Let $(M,g)$ be of class A. Then there exists a maximal ergodic measure  $\mu$ and $\e>0$ such that 
$$\dist (\supp \mu ,Light(M,[g]))\ge \e.$$
\end{cor}


\section{On the Definition of Causal Curves}

The notion of causal curves is best defined for spacetimes first. This represents no restriction since any Lorentzian manifold admits a time-orientable
twofold cover (\cite{ger0}). Assume that $(M,g)$ is time-oriented, i.e. $(M,g)$ is a spacetime.

We define what we understand by future and past pointing for geodesics first. A geodesic $\gamma$ of $(M,g)$ is future (past) pointing if 
$\dot{\gamma}$ is future (past) pointing. Note that this is well defined since $g(\dot{\gamma},\dot{\gamma})\equiv \text{const}$ and 
$g(\dot{\gamma}(t),X_{\gamma(t)})<(>)\;0$ for one $t$ if and only if $g(\dot{\gamma}(t),X_{\gamma(t)})<(>)\;0$ for all $t$. 

The following definition is taken from \cite{be}. A continuous curve $\gamma \colon I\to M$ is said to be future (past) pointing if for each $t_0\in I$ there 
exist an $\e >0$ and a convex normal neighborhood $U$ around $\gamma(t_0)$ with $\gamma|(t_0-\e,t_0+\e)\subseteq U$ such that given any 
$t_1< t_2\in (t_0-\e,t_0+\e)$ there is a future (past) pointing geodesic in $(U,g|_{U})$ connecting $\gamma (t_1)$ with $\gamma(t_2)$.

Call a curve in a spacetime causal if it is future or past pointing. The notion of causal curves (in opposition to future or past pointing) can be extended to 
general (possibly not time-oriented) Lorentzian manifolds via lifting: Let $(M,g)$ be a Lorentzian manifold and $\gamma\colon I\to M$ a continuous 
curve. We call $\gamma$ causal if the lift of $\gamma$ to a time-oriented cover $(M',g')$ of $(M,g)$ is  future or past pointing.

This definition does not depend on the chosen covering space or the chosen time orientation on the cover. Simply note that 
for any two time orientable covering spaces $(M',g')$ and $(M'',g'')$ of $(M,g)$ such that $(M'',g'')$ is a Lorentzian cover of 
$(M',g')$, any future (past) pointing curve in $(M',g')$ lifts to a future (past) pointing curve in $(M'',g'')$ and any future (past) 
pointing curve in $(M'',g'')$ projects to a future (past) pointing curve in $(M',g')$. Since the universal cover is simply connected, 
it is  time-orientable. Consequently the definition does not depend on the chosen time-orientable covering manifold. 

\begin{remark}[\cite{be}]\label{R-1}
Any future (past) pointing curve can be reparameterized to a Lipschitz continuous curve. Especially any future (past) 
pointing curve admits a monotone reparameterization w.r.t. $g_R$-arclength. This readily extends to causal curves.
\end{remark}

The following proposition is well known. We include the proof for the sake of completeness.

\begin{prop}\label{P00}
Let $\gamma\colon [a,b]\to M$ be a $g_R$-arclength parameterized curve.
Then $\gamma$ is future pointing if and only if $\dot{\gamma}(t)$ is future pointing for almost all $t\in [a,b]$.
\end{prop}

\begin{remark}
The definition does not depend on the chosen Riemannian metric. More precisely, the parameter change between two arclength parameterizations 
relative to two Riemannian metrics is a locally bi-Lipschitz map and the chain rule applies almost everywhere. Thus the fact that $\dot{\gamma}(t)$ is 
future pointing for almost all $t$ is independent of the particular arclength parameterization.

Following proposition \ref{P00} we could have defined future pointing curves as rectifiable curves $\gamma$ such that 
$\dot{\gamma}(s)$ is future pointing for almost all $s$, where $s\mapsto \gamma(s)$ is some parameter of $\gamma$
such that $\dot{\gamma}(s)$ exists almost everywhere.

At first sight, this definition may look more restrictive than the usual definition of future pointing, but is in fact 
equivalent. By remark \ref{R-1} any future pointing curve is rectifiable and therefore admits a 
$g_R$-arclength parameterization. Any $g_R$-arclength parameterization is Lipschitz. By Rademachers theorem 
any Lipschitz curve is differentiable almost everywhere, consequently any future pointing curve $\gamma$ admits a 
monotone reparameterization such that $\dot{\gamma}$ exists almost everywhere and is future pointing by proposition
\ref{P00}.
\end{remark}

\begin{proof}[Proof of proposition \ref{P00}]
(i) Assume that $\gamma$ is future pointing. Consider $t\in [a,b]$ such that $\dot{\gamma}(t)$ exists. Denote 
$p:=\gamma(t)$. Then for $|s-t|$ sufficiently small the curve $\widetilde{\gamma}(s):=\exp^{-1}_{p}(\gamma(s))$ is 
defined. By definition the vector $\widetilde{\gamma}(s)\in TM_p$ is future pointing in $(TM_{p},g_{p})$. Identify $T(TM_{p})_{0_p}$ with 
$TM_p$ in the canonical way. Then $\dot{\gamma}(t)=\dot{\widetilde{\gamma}}(t)$ is future pointing in $(TM_{p},g_{p})$ as a 
limit of future pointing vectors. Note that it cannot be $0$ since we assumed that $\gamma$ is parameterized by 
$g_R$-arclength. Since $\dot{\gamma}(t)$ exists for almost all $t\in[a,b]$, we obtain that $\dot{\gamma}(t)$ is future 
pointing for almost all $t\in [a,b]$.

(ii) Assume that $\dot{\gamma}(t)$ is future pointing for almost all $t\in [a,b]$. Let $s<t\in [a,b]$ be given 
such that $\gamma|_{[s,t]}\subset U$ for some open convex normal neighborhood $U$. We want to show that the 
uniquely defined geodesic $\zeta\colon [0,1]\to U$ connecting $\gamma(s)$ with $\gamma(t)$ is future pointing. 

Let $r<s$. Consider a  future pointing timelike geodesic $\xi\colon [r,s]\to U$ with terminal point $\xi(s)=\gamma(s)$. Then the curve $\sigma 
:=\exp_{\xi(r)}^{-1}\circ(\xi\ast \gamma|_{[s,t]})$ is well defined. Note that $\sigma$ is smooth on $[r,s]$. For $\tau\le s$ we know that$\sigma(\tau)
=\frac{\tau-r}{s-r}\dot{\xi}(r)$ is future pointing timelike in $(TM_{\xi(r)},g_{\xi(r)})$. This yields $g_{\xi(r)}(\sigma(s),\sigma(s))<0$. Consequently there 
exists $s_0>s$ such that $g_{\xi(r)}(\sigma(\tau),\sigma(\tau))<0$ for all $s\le \tau\le s_0$. Assume that there exists $s_0<t_0\le t$ with 
$g_{\xi(r)}(\sigma(s_0),\sigma(s_0))<g_{\xi(r)}(\sigma(t_0),\sigma(t_0))$. We can assume that $g_{\xi(r)}(\sigma(\tau),\sigma(\tau))<0$ for all 
$\tau\in [s_0,t_0]$, since $\tau \mapsto g_{\xi(r)}(\sigma(\tau),\sigma(\tau))$ is a continuous function. Then we know that $(\exp)_{\ast\;\sigma(\tau)}
(\sigma(\tau))$ is future pointing in $(TM_{\exp(\sigma(\tau))},g_{\exp(\sigma(\tau))})$ for all $\tau \in [s_0,t_0]$. For almost all $\tau\in [s_0,t_0]$ 
we have
\begin{align*}
\frac{d}{d\tau} g_{\xi(r)}(\sigma,\sigma)(\tau)&=2g_{\xi(r)}(\dot{\sigma}(\tau),\sigma(\tau))\\
&=2g_{\exp_{\xi(r)}(\sigma(\tau))}\left(\dot{\gamma}(\tau),\left(\exp_{\xi(r)}\right)_\ast(\sigma(\tau))\right)\le 0,
\end{align*}
using the Gau\ss\; lemma and the assumption that $\dot{\gamma}(\tau)$ is future pointing for almost all $\tau$. 
Then we get 
\begin{align*}
0<g_{\xi(r)}(\sigma(t_0),\sigma(t_0))-g_{\xi(r)}(\sigma(s_0),\sigma(s_0))
=\int_{s_0}^{t_0} \frac{d}{d\tau} g_{\xi(r)}(\sigma,\sigma)(\tau)d\tau \le 0.
\end{align*}
Therefore $\tau\mapsto g_{\xi(r)}(\sigma(\tau),\sigma(\tau))$ has to be monotone decreasing. This yields 
$$g_{\xi(r)}(\exp_{\xi(r)}^{-1}(\gamma(t)),\exp_{\xi(r)}^{-1}(\gamma(t)))<0$$ 
and the geodesic $\zeta_{r}\colon [0,1]\to M$, $\lambda\mapsto \exp_{\xi(r)}(\lambda\exp_{\xi(r)}^{-1}(\gamma(t)))$ 
is future pointing timelike. 

Now choose a sequence $\{r_n\}\subset [r,s]$ with $\lim r_n=s$ and 
geodesic $\zeta_{r_n}$ as above. The sequence $\zeta_{r_n}$ converges to the geodesic $\zeta$. Recall that the convergence
of $\zeta_{r_n}$ to $\zeta$ is equivalent to the convergence of $\dot{\zeta}_{r_n}(0)$ to $\dot{\zeta}(0)$. Since 
the set of future pointing vectors in $TM$ is closed and $\zeta$ is nonconstant, we see that $\zeta$ is future 
pointing. This construction is valid for any pair of parameters $s<t$ such that 
$\gamma(s)$ and $\gamma(t)$ are sufficiently close. Therefore we obtain that $\gamma$ is future pointing.
\end{proof}

At this point it is easy to see that for any causal curve $\gamma\colon [a,b]\to M$ in a Lorentzian manifold $(M,g)$ there exists 
a piecewise smooth causal curve $\gamma_p\colon [a,b]\to M$ such that $\gamma$ and $\gamma_p$ are homotopic 
with fixed endpoints via causal curves. It suffices to consider the case that $(M,g)$ is a spacetime, i.e. $\gamma$ is 
(w.l.o.g.) future pointing. Consider $s<t\in [a,b]$ such that $\gamma|_{[s,t]}$ is contained in a convex normal neighborhood
$U$ and let $\tau\in [s,t]$. By definition the unique geodesic $\zeta_\tau\colon [s,\tau]\to U$ connecting $\gamma(s)$ with 
$\gamma(\tau)$ is future pointing. Define the future pointing curve $\gamma_\tau:=\zeta_\tau\ast \gamma|_{[\tau,t]}$. Then
$\tau\mapsto\gamma_\tau$ defines a continuous deformation of $\gamma|_{[s,t]}$ into $\zeta_t$ via future pointing curves. 
Using a simple compactness argument we see that $\gamma$ is homotopic with fixed endpoints to a piecewise 
geodesic future pointing curve via future pointing curves.

The following proposition shows that actually more is true (We call a smooth curve timelike if it is causal and all tangents are 
timelike vectors).
\begin{prop}[\cite{pen1}]\label{P01}
Let $(M,g)$ be a Lorentzian manifold. For every causal curve $\gamma\colon [a,b]\to M$ there either exists a smooth 
timelike curve $\gamma_t\colon [a,b]\to M$ or a lightlike geodesic $\gamma_l\colon [a,b]\to M$ (i.e. $\dot{\gamma}_l$ 
lightlike) such that $\gamma_t$ and $\gamma_l$ are homotopic with fixed endpoints to $\gamma$ via causal curves.
\end{prop}
The proof relies essentially on the convexity of the set of future pointing vectors (in a time-orientable cover).

\end{document}